\documentclass[article,onefignum,onetabnum]{siamart250211}
\usepackage{_settings}

\usepackage{changes}
\definechangesauthor[name=Luiz, color=gray]{lmf}
\definechangesauthor[name=Pierre, color=blue]{pm}
\title{Convergence rates of curved boundary element methods for the 3D Laplace and Helmholtz equations\thanks{Submitted to the editors DATE.}}

\headers{BEM convergence rates for 3D Laplace and Helmholtz}{L. M. Faria, P. Marchand, H. Montanelli}

\ifpdf
\hypersetup{
  pdftitle={Convergence rates of curved boundary element methods for the 3D Laplace and Helmholtz equations},
  pdfauthor={L. M. Faria, P. Marchand, H. Montanelli}
}
\fi

\author{Luiz M. Faria\thanks{POEMS, CNRS, Inria, ENSTA, Institut Polytechnique de Paris, 91120 Palaiseau, France.}
\and Pierre Marchand\footnotemark[2]
\and Hadrien Montanelli\thanks{Inria, Unit\'{e} de Math\'{e}matiques Appliqu\'{e}es, ENSTA, Institut Polytechnique de Paris, 91120 Palaiseau, France.}}

\begin{document}

\maketitle

\begin{abstract}
  We establish improved convergence rates for curved boundary element methods applied to the
  three-dimensional (3D) Laplace and Helmholtz equations with smooth geometry and data. Our analysis relies on a
  precise {\color{black} estimate} of the consistency errors introduced by the perturbed bilinear and sesquilinear forms. We illustrate our results with numerical
  experiments in 3D based on basis functions and curved triangular elements up to order four.
\end{abstract}



\section{Introduction}

The Laplace and Helmholtz equations play a fundamental role in mathematical physics and engineering. The Laplace equation governs steady-state phenomena such as electrostatics, gravitation, and {\color{black}potential} fluid flow, while the Helmholtz equation models time-harmonic wave propagation in acoustics, electromagnetics, and elasticity. Solutions to these equations are often efficiently computed by solving the associated boundary integral equations using boundary element methods, which are particularly effective for problems in unbounded domains or with complex geometries. Understanding the accuracy and convergence rates of these methods is crucial for reliable simulations in scientific and industrial applications. The goal of this paper is to provide a detailed analysis of these rates and to support our findings with numerical experiments.

The numerical error in boundary element methods arises from two sources: (i) the
approximation of functions using polynomials of degree $m\geq0$, and (ii) the discretization
of the geometry through elements of order $\ell\geq1$. Since N\'{e}d\'{e}lec's pioneering
work in 1976 \cite{nedelec1976}, progress in this area has been limited, with only a handful
of works---mostly by N\'{e}d\'{e}lec himself and his student Giroire in the late 1970s and
early 1980s \cite{giroire1978b, giroire1982, giroire1978a, nedelec1977}. Related results can also be found in the monograph by Sauter and
Schwab \cite{sauter2011}, in Bendali's work~\cite{bendali1984}, and
in Christiansen's thesis \cite{christiansen2002}. The classical 1976 results of
N\'{e}d\'{e}lec are mentioned in the work of Wendland \cite{wendland1981, wendland1983}, who later extended it to more general elliptic operators \cite{wendland1990}, and in {\color{black}Dautray} and Lions \cite[Chap.~XIII]{dautray1990}. These results are practically important for selecting $m$ and $\ell$ such that the two errors decay at the same rate.

In this paper, we show that for smooth geometries and incident fields, the pointwise error in the numerical solution satisfies
\begin{align}
   & \vert u(\bs{x}) - u_h(\bs{x})\vert \leq c_{\bs{x}} \left(h^{2m+3} + h^{\ell+1}\right), \qquad \text{(Laplace/Helmholtz with single-layer)}, \label{eq:pointwise_single} \\
   & \vert u(\bs{x}) - u_h(\bs{x})\vert \leq c_{\bs{x}} \left(h^{2m+2} + h^{\ell}\right), \qquad \text{(Laplace/Helmholtz with double-layer or CFIE)}, \label{eq:pointwise_double}
\end{align}
where the constant $c_{\bs{x}}>0$ depends on the off-surface evaluation point $\bs{x}$ but not on the mesh size~$h$. These convergence rates, which were previously observed in numerical experiments~\cite{montanelli2022, montanelli2024a, pascal2024}, improve upon existing results in the literature, as summarized in \cref{tab:results}. \textcolor{black}{Specifically, compared with~\cite{sauter2011}, the geometric error improves from \(h^{\ell+1/2}\) to \(h^{\ell+1}\) for the single-layer formulation, and the factor \(\vert\log h\vert\) is removed for the double-layer and CFIE formulations; compared with~\cite{nedelec1976}, the approximation error for the Laplace single-layer improves from \(h^{m+2}\) to \(h^{2m+3}\).} We complement these theoretical rates with extensive numerical experiments using piecewise polynomials and curved triangular meshes up to order four---an investigation that, to the best of our knowledge, has not been previously carried out. We also observe the $h^{\ell+2}$ superconvergence behavior previously reported in~\cite{montanelli2022, montanelli2024a} for even values of~$\ell$. Similar geometric superconvergence for quadratic meshes has also been observed in other contexts~\cite{bonito2018, caubet2024}.

\begin{table}
   \caption{\textit{The published results of N\'ed\'elec and Giroire \cite{giroire1982,nedelec1976} are either not sharp (Laplace single-layer) or provided without proof (Helmholtz single-layer). Sharper results appear in unpublished technical reports \cite{giroire1978b,nedelec1977}, which are not available online. These match \cref{eq:pointwise_single,eq:pointwise_double}, but the proofs are either incomplete (including for the Helmholtz single-layer and Laplace double-layer) or nonexistent (Helmholtz double-layer). {\color{black}For the double-layer, }their analysis uses the interpolated normal, {\color{black}obtained by interpolating the exact normal at the nodes}; however, the results shown here have been adjusted to correspond to the element normal. The estimates in Sauter and Schwab's book \cite{sauter2011} follow from the coefficients in \cite[Tab.~8.2]{sauter2011}, evaluated explicitly in \cite[Cor.~8.2.9]{sauter2011}.}} 
  \centering
  \ra{1.3}
  \begin{tabular}{c|cc}
    \toprule
                                  & Laplace                                                                   & Helmholtz                                                                  \\
    \midrule
    \multirow{3}{*}{Single-layer} & \(h^{m+2} + h^{\ell+1}\), published, proved \cite{nedelec1976}            & \cref{eq:pointwise_single}, published, no proof \cite{giroire1982}         \\
                                  &                                                                           & \cref{eq:pointwise_single}, unpublished, partial proof \cite{giroire1978b} \\
                                  & \(h^{2m+3} + h^{\ell+1/2}\) \cite{sauter2011}                             & \(h^{2m+3} + h^{\ell+1/2}\) \cite{sauter2011}                              \\
    \midrule
    \multirow{2}{*}{Double-layer} & \cref{eq:pointwise_double}, unpublished, partial proof \cite{nedelec1977} & \cref{eq:pointwise_double}, unpublished, no proof \cite{giroire1978b}      \\
                                  & \(h^{2m+2} + h^\ell\vert \log h\vert\) \cite{sauter2011}                  & \(h^{2m+2} + h^\ell\vert\log h\vert\) \cite{sauter2011}                    \\
    \bottomrule
  \end{tabular}
  \label{tab:results}
\end{table}

We focus exclusively on the Dirichlet problem. Sobolev spaces of order $s$ over a domain $\Omega$ and its boundary $\Gamma$ are denoted by $H^s(\Omega)$ and $H^s(\Gamma)$, respectively, with the corresponding norm on $H^s(\Gamma)$ written as $\|\cdot\|_s$. Throughout, $c > 0$ denotes a generic constant that may vary from line to line. To simplify the exposition, we have chosen to remain concrete in our presentation, considering only the single- and double-layer operators, as well as their linear combination, for the Laplace and Helmholtz equations. While we expect many of the results to extend to more general settings, this focused approach has proven effective in deriving our convergence rates.

\section{Boundary integral operators and discretization}\label{sec:theory}

In this section, we present the tools needed, keeping the discussion as concrete as possible.

\subsection{Solving the Laplace equation with integral operators}

We start with the Laplace equation. To simplify the exposition and avoid the introduction of weighted Sobolev spaces, we will focus on the \textit{interior} problem.

Let $\Omega\subset\R^3$ be an open, bounded set with a smooth (i.e., $C^\infty$) boundary $\Gamma$, and $f\in H^{1/2}(\Gamma)$.

\begin{problem}[Laplace equation]\label{pb:strong-laplace}
Find $u\in H^1(\Omega)$ such that
\begin{align*}
  \left\{
  \begin{array}{ll}
    \Delta u = 0 \quad \text{in $\Omega$}, \\[0.4em]
    u = f \quad \text{on $\Gamma$}.
  \end{array}
  \right.
\end{align*}
\end{problem}
There is a unique solution to \cref{pb:strong-laplace} \cite{mclean2000}---theoretical results for both the interior and
exterior problems using integral equations go back to N\'{e}d\'{e}lec and Planchard in 1973
\cite[Lem.~1.1]{nedelec1973}. Moreover, if $f\in H^{m+1/2}(\Gamma)$
then $u\in H^{m+1}(\Omega)$ for all integer $m\geq0$ \cite[Thm.~1.2]{nedelec1973}.

We can look for the solution $u$ as a single- or double-layer potential. We start with the former.

\begin{problem}[Single-layer]\label{pb:laplace-SL}
Find $p\in H^{-1/2}(\Gamma)$ such that
\begin{align*}
  S_0p = f \quad \text{in $H^{1/2}(\Gamma)$},
\end{align*}
with $S_0:H^{-1/2}(\Gamma)\to H^{1/2}(\Gamma)$ defined by
\begin{align}\label{eq:laplace-SL}
  (S_0p)(\bs{x}) = \frac{1}{4\pi}\int_\Gamma\frac{p(\bs{y})}{\vert\bs{x}-\bs{y}\vert}d\Gamma(\bs{y}), \quad \bs{x}\in\Gamma.
\end{align}
The solution $u$ in $\Omega$ reads $u=\mathcal{S}_0p$ with $\mathcal{S}_0:H^{-1/2}(\Gamma)\to H^1(\Omega)$ defined by \cref{eq:laplace-SL} for $\bs{x}\in\Omega$.
\end{problem}

We note that $S_0$ is an isomorphism between $H^{-1/2}(\Gamma)$ and $H^{1/2}(\Gamma)$,\footnote{By an isomorphism, we mean a linear, bounded, bijective map whose inverse is also bounded.} and more generally between $H^s(\Gamma)$ and $H^{s+1}(\Gamma)$ for any $s\in\R$ \cite[Thm.~1.2]{nedelec1973}; see also \cite[Thm.~6.34]{steinbach2008}.

We continue with the double-layer potential.

\begin{problem}[Double-layer]\label{pb:laplace-DL}
Find $p\in H^{1/2}(\Gamma)$ such that
\begin{align*}
  \left(\frac{I}{2} - D_0\right)p = f \quad \text{in $H^{1/2}(\Gamma)$},
\end{align*}
with $D_0:H^{1/2}(\Gamma)\to H^{1/2}(\Gamma)$ defined by
\begin{align}\label{eq:laplace-DL}
  (D_0p)(\bs{x}) = \frac{1}{4\pi}\int_\Gamma \frac{\partial }{\partial \bs{n}(\bs{y})} \left(\frac{1}{\vert\bs{x}-\bs{y}\vert}\right) p(\bs{y}) d\Gamma(\bs{y}) = \frac{1}{4\pi}\int_\Gamma \frac{(\bs{x}-\bs{y}) \cdot \bs{n}(\bs{y})}{\vert\bs{x}-\bs{y}\vert^3} p(\bs{y}) d\Gamma(\bs{y}), \quad \bs{x}\in\Gamma,
\end{align}
where $\bs{n}(\bs{y})$ denotes the unit normal vector pointing outwards from $\Omega$ at the point $\bs{y}$.
The solution $u$ in $\Omega$ reads $u=-\mathcal{D}_0p$ with $\mathcal{D}_0:H^{1/2}(\Gamma)\to H^1(\Omega)$ defined by \cref{eq:laplace-DL} for $\bs{x}\in\Omega$.
\end{problem}

Here, the operator \((I/2 - D_0)\) is an isomorphism from \(H^{1/2}(\Gamma)\) to itself, and more generally from \(H^s(\Gamma)\) to \(H^s(\Gamma)\) for any \(s \in \mathbb{R}\), since \(\Gamma\) is smooth; see \cite[Thm.~6.34]{steinbach2008}. Therefore, the problem is also well-posed in \(L^2(\Gamma)\), which is the space we choose for our variational formulation in \cref{sec:laplace}. This choice is motivated by the numerical complexity of implementing the \(H^{1/2}(\Gamma)\)-inner product. Since \(f\in H^{1/2}(\Gamma)\), the solutions obtained in \(H^{1/2}(\Gamma)\) and in \(L^2(\Gamma)\) coincide. See also \cite[Rem.~3.8.12]{sauter2011} and \cite[Thm.~2.25]{chandler2012} for related discussions.

\subsection{Solving the Helmholtz equation with integral operators}

We continue with the Helmholtz equation; we will focus on the \textit{exterior} problem. The
functions are now complex-valued. Let $\Omega\subset\R^3$ be an open, bounded set with a smooth (i.e., $C^\infty$)
boundary $\Gamma$, and $f\in H^{1/2}(\Gamma)$. Let $k>0$ be the wavenumber.

\begin{problem}[Helmholtz equation]\label{pb:strong-helmholtz}
Find $u\in H^1_\mrm{loc}(\R^3\setminus{\color{black}\overline{\Omega}})$ such that
\begin{align*}
  \left\{
  \begin{array}{ll}
    \Delta u + k^2 u = 0 \quad \text{in $\R^3\setminus\overline{\Omega}$}, \\[0.4em]
    u = f \quad \text{on $\Gamma$},                                        \\[0.4em]
    \text{$u$ is radiating}.
  \end{array}
  \right.
\end{align*}
\end{problem}
The (Sommerfeld) radiation condition in \cref{pb:strong-helmholtz} reads
\begin{align*}
  \lim_{r\to\infty} r\left(\frac{\partial {\color{black}u}}{\partial r} - ik{\color{black}u}\right) = 0, \quad r = \vert\bs{x}\vert \quad \text{(uniformly in $\bs{x}/\vert\bs{x}\vert$)}.
\end{align*}
There is a unique solution to \cref{pb:strong-helmholtz} \cite{mclean2000}. Moreover, if $f\in H^{m+1/2}(\Gamma)$ then $u\in H^{m+1}_\mrm{loc}(\R^3\setminus{\color{black}\overline{\Omega}})$ for all integer $m\geq0$ \cite{mclean2000}.

We start with the single-layer potential.

\begin{problem}[Single-layer]\label{pb:helmholtz-SL}
Find $p\in H^{-1/2}(\Gamma)$ such that
\begin{align*}
  Sp = f \quad \text{in $H^{1/2}(\Gamma)$},
\end{align*}
with $S:H^{-1/2}(\Gamma)\to H^{1/2}(\Gamma)$ defined by
\begin{align}\label{eq:helmholtz-SL}
  (Sp)(\bs{x}) = \frac{1}{4\pi}\int_\Gamma \frac{e^{ik \vert\bs{x}-\bs{y}\vert}}{\vert\bs{x}-\bs{y}\vert} p(\bs{y}) d\Gamma(\bs{y}){\color{black}, \quad \bs{x}\in\Gamma}.
\end{align}
The solution $u$ in $\R^3\setminus{\color{black}\overline{\Omega}}$ reads $u=\mathcal{S}p$ with $\mathcal{S}:H^{-1/2}(\Gamma)\to H^1_\mrm{loc}(\R^3\setminus{\color{black}\overline{\Omega}})$ defined by \cref{eq:helmholtz-SL} \textcolor{black}{for $\bs{x}\in\R^3\setminus{\overline{\Omega}}$}.
\end{problem}

If $k^2$ is not a Dirichlet eigenvalue of $-\Delta$ in $\Omega$, then the operator $S$ is an isomorphism between $H^s(\Gamma)$ and $H^{s+1}(\Gamma)$ for any $s\in\R$ \cite[Thm.~2]{giroire1982}; see also \cite[Thm.~6.34]{steinbach2008}.

We continue with the double-layer potential.

\begin{problem}[Double-layer]\label{pb:helmholtz-DL}
Find $p\in H^{1/2}(\Gamma)$ such that
\begin{align*}
  \left(\frac{I}{2} + D\right)p = f \quad \text{in $H^{1/2}(\Gamma)$},
\end{align*}
with $D:H^{1/2}(\Gamma)\to H^{1/2}(\Gamma)$ defined by
\begin{align}\label{eq:helmholtz-DL}
  (Dp)(\bs{x}) & = \frac{1}{4\pi}\int_\Gamma \frac{\partial }{\partial \bs{n}(\bs{y})} \left(\frac{e^{ik\vert\bs{x}-\bs{y}\vert}}{\vert\bs{x}-\bs{y}\vert}\right) p(\bs{y}) d\Gamma(\bs{y}), \nonumber  \\
               & = \frac{1}{4\pi}\int_\Gamma(1-ik\vert\bs{x}-\bs{y}\vert)\frac{e^{ik\vert\bs{x}-\bs{y}\vert}}{\vert\bs{x}-\bs{y}\vert^3}(\bs{x}-\bs{y}) \cdot \bs{n}(\bs{y}) p(\bs{y}) d\Gamma(\bs{y}), \quad {\color{black}\bs{x}\in\Gamma,}
\end{align}
where, once again, $\bs{n}(\bs{y})$ denotes the unit normal vector pointing outwards from $\Omega$ at the point $\bs{y}$. The solution $u$ in $\R^3\setminus{\color{black}\overline{\Omega}}$ reads $u=\mathcal{D}p$ with
$\mathcal{D}:H^{1/2}(\Gamma)\to H^1_\mrm{loc}(\R^3\setminus{\color{black}\overline{\Omega}})$ defined by \cref{eq:helmholtz-DL} \textcolor{black}{for $\bs{x}\in\R^3\setminus{\overline{\Omega}}$}.
\end{problem}

Note that $(I/2+D)$ is an isomorphism between $H^s(\Gamma)$ and $H^s(\Gamma)$ for any $s\in\R$ \cite[Thm.~6.34]{steinbach2008}, as long as $k^2$ is not a Neumann eigenvalue of $-\Delta$ in $\Omega$. We will solve this problem in $L^2(\Gamma)$.

The single- and double-layer potentials can be combined to form the \textit{Combined Field Integral Equation (CFIE)} \cite{brakhage1965}, for some real scalar $\eta>0$. (In practice, one often chooses $\eta=k$.)

\begin{problem}[CFIE]\label{pb:helmholtz-CFIE}
Find $p\in H^{1/2}(\Gamma)$ such that
\begin{align*}
  \left(\frac{I}{2} + D - i\eta S\right)p = f \quad \text{in $H^{1/2}(\Gamma)$},
\end{align*}
with \textcolor{black}{$S:H^{-1/2}(\Gamma)\to H^{1/2}(\Gamma)$, composed with the compact embedding $H^{1/2}(\Gamma)\hookrightarrow H^{-1/2}(\Gamma)$,} and $D:H^{1/2}(\Gamma)\to H^{1/2}(\Gamma)$ defined by \cref{eq:helmholtz-SL} and \cref{eq:helmholtz-DL}. The solution $u$ in $\R^3\setminus\overline{\Omega}$ reads $u=(\mathcal{D}-i\eta\mathcal{S})p$ with the corresponding $\mathcal{S}$ and $\mathcal{D}$.
\end{problem}

Here, $(I/2+D-i\eta S)$ is an isomorphism between $H^s(\Gamma)$ and $H^s(\Gamma)$ for any $s\in\R$ \cite[Thm.~6.34]{steinbach2008}. We will also solve this problem in $L^2(\Gamma)$.

A summary of the boundary integral formulations introduced above is provided in \cref{tab:integral-operators}.

\begin{table}
  \caption{\textit{Summary of boundary integral formulations and solution representations for Laplace (interior) and Helmholtz (exterior) problems. Columns correspond to single-layer, double-layer, and CFIE formulations. Each row shows: (i) the problem label, (ii) the boundary integral equation, and (iii) the solution representation in the domain.}}
  \centering
  \ra{1.3}
  \begin{tabular}{c|ccc}
    \toprule
                                          & Single-layer           & Double-layer               & CFIE                                     \\
    \midrule
    \multirow{3}{*}{Laplace (interior)}   & \cref{pb:laplace-SL}   & \cref{pb:laplace-DL}       &                                          \\
                                          & $S_0p = f$             & $(\frac{I}{2} - D_0)p = f$ & ---                                      \\
                                          & $u = \mathcal{S}_0p$
                                          & $u = -\mathcal{D}_0p$                                                                          \\
    \midrule
    \multirow{3}{*}{Helmholtz (exterior)} & \cref{pb:helmholtz-SL} & \cref{pb:helmholtz-DL}     & \cref{pb:helmholtz-CFIE}                 \\
                                          & $Sp = f$               & $(\frac{I}{2} + D)p = f$   & $(\frac{I}{2} + D - i\eta S)p = f$       \\
                                          & $u = \mathcal{S}p$     & $u = \mathcal{D}p$         & $u = (\mathcal{D} - i\eta \mathcal{S})p$ \\
    \bottomrule
  \end{tabular}
  \label{tab:integral-operators}
\end{table}

\subsection{Methodology}

We study the approximation of all previously introduced boundary integral equations using a Galerkin discretization of their weak formulations. Standard tools required to establish the well-posedness of these variational formulations---as well as of their discretized counterparts, both with and without geometric approximation---are recalled in \cref{app:theory}. These include the coercive case (the single-layer potential for the interior Laplace problem) and the ``coercive plus compact'' case (applicable to the remaining problems). The assumptions and theoretical results are summarized in \cref{tab:theory}.

In both settings, we obtain a Strang-type estimate (see \cref{lem:strang-A,lem:strang-B}), which bounds the discretization error by the sum of a Galerkin approximation error (depending on the quality of the discrete approximation space), a consistency error arising from the geometric approximation in the bilinear or sesquilinear form, and a consistency error in the right-hand side. We do not account for the quadrature error introduced in the numerical evaluation of forms.

\begin{table}
  \caption{\textit{We outline the key assumptions for existence and uniqueness in the weak formulation, as well as for the Galerkin and perturbed Galerkin approximation problems. In the coercive case, the discrete coercivity directly follows from the coercivity at the continuous level, while the uniform coercivity can be derived from both coercivity and consistency. In the ``coercive plus compact'' case, the discrete inf-sup conditions follow from coercivity and compactness, while the uniform discrete conditions can be obtained from the discrete ones and consistency.}}
  \centering
  \ra{1.3}
  \begin{tabular}{c|cc}
    \toprule
                                        & Coercive case                         & ``Coercive plus compact'' case   \\
    \midrule
    \multirow{2}*{Weak formulation}     & coercivity                            & coercivity \& compactness        \\
                                        & \cref{thm:lax-milgram} (Lax--Milgram) & \cref{thm:fredholm} (Fredholm)   \\
    \midrule
    \multirow{2}{*}{Galerkin}           & discrete coercivity                   & discrete inf-sup cond.           \\
                                        & \cref{lem:cea} (C\'{e}a)              & \cref{lem:babuska} (Babu\v{s}ka) \\
    \midrule
    \multirow{2}{*}{Perturbed Galerkin} & uniform coercivity                    & uniform discrete inf-sup cond.   \\
                                        & \cref{lem:strang-A} (Strang)          & \cref{lem:strang-B} (Strang)     \\
    \bottomrule
  \end{tabular}
  \label{tab:theory}
\end{table}

\subsection{Surface discretization}\label{subsec:geometry}

The smooth surface \(\Gamma\) is approximated by a sequence of piecewise polynomial surfaces \(\{\Gamma_h\}_{h>0}\), where each \(\Gamma_h\) consists of {\color{black} curved} triangular elements of polynomial degree \(\ell \geq 1\), and \(h > 0\) denotes the mesh size parameter measuring the maximal diameter of the elements in \(\Gamma_h\). (For an explicit construction in the quadratic case, see \cite{montanelli2022}.) We assume that the sequence \(\{\Gamma_h\}_{h>0}\) is \textit{shape-regular} \textcolor{black}{and \textit{quasi-uniform}}; see \cite[Def.~11.2]{ern2021a} or \cite[Rem.~4.1.14]{sauter2011}.

\textcolor{black}{We define the tubular neighborhood 
\begin{align*}
  \mathcal{N} = \left\{\bs{x} + t\,\bs{n}(\bs{x}) \;:\; \bs{x}\in\Gamma,\; t\in\R \; \text{with} \; \vert t\vert<\delta\right\},
\end{align*}
where $\delta>0$ is small enough so that the orthogonal projection $\Psi:\mathcal{N}\to\Gamma$ is well-defined and smooth.}
Under the standard assumption that \(\Gamma_h\) converges to \(\Gamma\) as \(h \to 0\) in a sufficiently smooth manner, there exists \(h_0 > 0\) such that the projection $\Psi_h := \Psi|_{\Gamma_h}:\Gamma_h\to\Gamma$ is bijective and smooth, \textcolor{black}{and \(\Gamma_h\subset\mathcal{N}\) for all \(h \leq h_0\)}. We denote by \(\Psi_h^{-1} : \Gamma \to \Gamma_h\) the inverse map (pullback), and by \(J_h^{-1} : \Gamma \to \mathbb{R}\) its Jacobian determinant. See \cite{nedelec1976} and \cref{app:geometry} for details. 

\textcolor{black}{The map \(\Psi_h\) is mainly an analysis tool and is rarely available in practice. As a consequence, data given only on \(\Gamma\), such as \(f\), cannot be transferred to \(\Gamma_h\) by composition with \(\Psi_h\). We therefore assume that \(f = F|_{\Gamma}\) for some \(F \in W^{1,\infty}(\mathcal{N})\) defined in the tubular neighborhood \(\mathcal{N}\). This is the case in scattering by a plane wave \(u_{\text{inc}}(\boldsymbol{x}) = e^{ik\boldsymbol{x}\cdot\boldsymbol{d}}\), where \(F = -u_{\text{inc}}\) is smooth on all of \(\mathbb{R}^3\) and satisfies \(\|\nabla F\|_{L^\infty(\mathcal{N})} = k\). The discrete problems then use the data \(F|_{\Gamma_h}\), and evaluating \(F\) on \(\Gamma_h\) rather than on \(\Gamma\) introduces an additional geometric error, which is quantified in \cref{lem:rhs_consistency}.}

\subsection{Finite element spaces}\label{subsec:FEM}

We now define the discrete finite element spaces used to approximate the boundary integral operators.

\paragraph{Single-layer potential} The natural function space is \(H^{-1/2}(\Gamma)\). Let \(\{V_h\}_{h > 0}\) be a family of finite-dimensional subspaces of \(H^{-1/2}(\Gamma_h)\). We define the lifted discrete spaces
\begin{align*}
  \hat{V}_h = \left\{ \hat{p}_h = p_h \circ \Psi^{-1}_h \,\middle|\, p_h \in V_h \right\} \subset H^{-1/2}(\Gamma).
\end{align*}
For \(m \geq 1\), we take \(V_h\) to be the space of continuous piecewise polynomial functions of degree \(m\) on \(\Gamma_h\), commonly called continuous Lagrange finite elements. These functions are uniquely determined by their values at nodal interpolation points (e.g., vertices, edge midpoints, and possibly interior points) and are globally continuous across element boundaries. \textcolor{black}{For \(m = 0\), we take \(V_h\) to be the space of discontinuous Lagrange finite elements of degree zero, i.e., piecewise constant functions, which are discontinuous across element boundaries.} For details, we refer to \cite{montanelli2022} for the quadratic case and \cite{ern2021a, sauter2011} for the general case. The functions in \(\hat{V}_h\) are lifted versions of those in \(V_h\); though generally non-polynomial, they share the same smoothness on each element and the same continuity across elements since \(\Psi_h^{-1}\) is smooth on each element.

\paragraph{Double-layer potential} The natural function space is \(L^2(\Gamma)\). Let \(\{V_h\}_{h > 0}\) be a family of finite-dimensional subspaces of \(L^2(\Gamma_h)\). We define the lifted discrete spaces
\begin{align*}
  \hat{V}_h = \left\{ \hat{p}_h = p_h \circ \Psi^{-1}_h \,\middle|\, p_h \in V_h \right\} \subset L^2(\Gamma).
\end{align*}
{\color{black} Again, we take \(V_h\) to consist of continuous Lagrange finite elements of degree \(m \geq 1\), or discontinuous Lagrange finite elements of degree \(m = 0\), on \(\Gamma_h\).

\paragraph{Real and complex cases} The Lagrange basis functions of \(V_h\) are real-valued. For the Laplace equation, \(V_h\) is the real span of these basis functions. For the Helmholtz equation, the functions are complex-valued, and \(V_h\) is their span with complex coefficients.}

\subsection{Consistency of the right-hand side}

{\color{black}
Throughout this subsection, \(V_h\) denotes any of the discrete spaces of \cref{subsec:FEM}, for either the single-layer or the double-layer formulation, with real coefficients for the Laplace equation and complex coefficients for the Helmholtz equation. Accordingly, \((\cdot,\cdot)\) and \((\cdot,\cdot)_h\) denote the real or complex \(L^2(\Gamma)\)- and \(L^2(\Gamma_h)\)-inner products,
\begin{align*}
  (u,v) = \int_{\Gamma} u(\bs{x})\overline{v(\bs{x})}\,d\Gamma(\bs{x}), \qquad (u,v)_h = \int_{\Gamma_h} u(\bs{x})\overline{v(\bs{x})}\,d\Gamma_h(\bs{x}),
\end{align*}
where the complex conjugate is omitted in the real case. In all the considered problems, we define \(f_h = s_h(F|_{\Gamma_h}) \in V_h\) as the \(L^2(\Gamma_h)\)-projection of \(F|_{\Gamma_h}\) onto \(V_h\), that is, \((F|_{\Gamma_h} - f_h, q_h)_h = 0\) for all \(q_h \in V_h\), and \(\hat{f}_h = f_h\circ\Psi^{-1}_h\). Thus, the error on the right-hand side can be measured independently of the considered problem.

\begin{lemma}[Right-hand side consistency]\label{lem:rhs_consistency}
  Let \(V_h\) be any of the discrete spaces above, real or complex. There exists \(h_0>0\) such that for all \(F \in W^{1,\infty}(\mathcal{N})\) with \(f = F|_{\Gamma}\), all \(h\leq h_0\), and all \(\hat{q}_h\in\hat{V}_h\),
  \begin{align*}
    \lvert(f - \hat{f}_h J_h^{-1},\hat{q}_h)\rvert \leq c h^{\ell+1} (\lVert f\rVert_{0} + \lVert \nabla F \rVert_{L^{\infty}(\mathcal{N})}) \lVert \hat{q}_h \rVert_0.
  \end{align*}
\end{lemma}}

{\color{black}
\begin{proof}
  Let \(q_h = \hat{q}_h\circ\Psi_h \in V_h\). We first use the fact that \(f_h\) is the \(L^2(\Gamma_h)\)-projection of \(F|_{\Gamma_h}\):
  \begin{align*}
    (f - \hat{f}_h J_h^{-1},\hat{q}_h) = (f - F|_{\Gamma_h} \circ\Psi_h^{-1}J_h^{-1},\hat{q}_h) + (F|_{\Gamma_h} - f_h ,q_h)_h =  (f - F|_{\Gamma_h} \circ\Psi_h^{-1}J_h^{-1},\hat{q}_h).
  \end{align*}
  It remains to estimate the geometric error due to the evaluation of \(F\) on \(\Gamma_h\) instead of \(\Gamma\):
  \begin{align*}
    \lvert (f - \hat{f}_h J_h^{-1},\hat{q}_h)\rvert &\leq \lvert (f - F|_{\Gamma_h} \circ\Psi_h^{-1},\hat{q}_h) \rvert + \lvert( F|_{\Gamma_h} \circ\Psi_h^{-1} ( 1- J_h^{-1}),\hat{q}_h) \rvert, \\
    & \leq \lVert \hat{q}_h \rVert _0 (\lVert f - F|_{\Gamma_h} \circ\Psi_h^{-1} \rVert_0 + c h^{\ell+1}\lVert F|_{\Gamma_h} \circ\Psi_h^{-1} \rVert_0 ), 
  \end{align*}
  where we used \cref{lem:geometry}. We conclude noting that \(\lVert f- F|_{\Gamma_h}\circ\Psi_h^{-1} \rVert_0 \leq c h^{\ell+1} \lVert \nabla F \rVert_{L^{\infty}(\mathcal{N})}\) since for every \(\bs{x}\in\Gamma\) the segment joining \(\bs{x}\) and \(\Psi_h^{-1}(\bs{x})\in \Gamma_h\) is included in \(\mathcal{N}\) (see~\cref{subsec:geometry}), and \(\lvert \bs{x} - \Psi_h^{-1}(\bs{x})\rvert\leq c h^{\ell+1}\) via \cref{lem:geometry}. Finally,
  \begin{align*}
    \lVert F|_{\Gamma_h} \circ\Psi_h^{-1} \rVert_0 \leq \lVert f \rVert_0 + \lVert f - F|_{\Gamma_h}\circ\Psi_h^{-1}\rVert_0 \leq \lVert f \rVert_0 + c h^{\ell+1} \lVert \nabla F \rVert_{L^{\infty}(\mathcal{N})}.
  \end{align*}
\end{proof}}

{\color{black}
Note that \(F \in W^{1,\infty}(\mathcal{N})\) implies \(f \in W^{1,\infty}(\Gamma) \subset H^{1/2}(\Gamma)\), so \cref{lem:rhs_consistency} applies to the data of all the considered problems. We will often use \cref{lem:rhs_consistency} with \(\lVert f\rVert_1\), since \(\lVert f\rVert_0 \leq \lVert f\rVert_1\).}

\subsection{Summary of the theoretical results}

{\color{black}
The results proved in \cref{sec:laplace,sec:helmholtz} are listed in \cref{tab:summary-results}. For every formulation, the analysis follows the same steps: a consistency estimate for the perturbed bilinear or sesquilinear form, an error estimate in the intrinsic norm via Strang's lemma, an estimate in a stronger norm via inverse inequalities, an estimate in weaker norms via a duality argument, and finally a pointwise estimate away from \(\Gamma\). These estimates read
\begin{align*}
  & \Vert p - \hat{p}_h\Vert_{-1/2} \leq c\left(h^{m+3/2} + h^{\ell+1/2}\right), \qquad && \text{(intrinsic norm)}, \\
  & \Vert p - \hat{p}_h\Vert_{0} \leq c\left(h^{m+1} + h^{\ell}\right), \qquad && \text{(stronger norm)}, \\
  & \Vert p - \hat{p}_h\Vert_{-m-2} \leq c\left(h^{2m+3} + h^{\ell+1}\right), \qquad && \text{(weaker norm)},
\end{align*}
for the single-layer formulations, and
\begin{align*}
  & \Vert p - \hat{p}_h\Vert_{0} \leq c\left(h^{m+1} + h^{\ell}\right), \qquad && \text{(intrinsic norm)}, \\
  & \Vert p - \hat{p}_h\Vert_{\epsilon} \leq c\left(h^{m+1-\epsilon} + h^{\ell-\epsilon}\right) \quad \forall\epsilon\in(0,1/2), \qquad  && \text{(stronger norm)}, \\
  & \Vert p - \hat{p}_h\Vert_{-m-1} \leq c\left(h^{2m+2} + h^{\ell}\right), \qquad && \text{(weaker norm)},
\end{align*}
for the double-layer and CFIE formulations, which yield the pointwise estimates \cref{eq:pointwise_single,eq:pointwise_double}.

The approximation errors \(h^{2m+3}\) and \(h^{2m+2}\) are classical. Our gains lie in the geometric errors of the weaker-norm and pointwise estimates, and they come from sharper consistency estimates. For the single-layer formulations, consistency in \(L^2(\Gamma)\times L^2(\Gamma)\) (\cref{lem:consistency-laplace-SL,lem:consistency-helmholtz-SL}) rather than in \(L^2(\Gamma)\times H^{-1/2}(\Gamma)\) gives a geometric error of order \(h^{\ell+1}\) instead of \(h^{\ell+1/2}\)~\cite{sauter2011}, while the duality argument recovers the full approximation order \(h^{2m+3}\) instead of \(h^{m+2}\)~\cite{nedelec1976}. For the double-layer and CFIE formulations, writing the forms in terms of \(\hat{p}_h(\bs{x}) - \hat{p}_h(\bs{y})\), as in \cref{eq:laplace_bilinear_form_rewritten}, yields consistency in \(H^{\epsilon}(\Gamma)\times L^2(\Gamma)\) (\cref{lem:consistency-laplace-DL,lem:consistency-helmholtz-DL,lem:consistency-helmholtz-CFIE}), which removes the \(\vert\log h\vert\) factor of~\cite{sauter2011}. 

\begin{table}
  \caption{\textit{Theoretical results for each formulation. Wavenumber-explicit results are given in \cref{lem:consistency-helmholtz-CFIE-k-explicit,thm:intrinsic-norm-helmholtz-CFIE-k-explicit}.}}
  \centering
  \ra{1.3}
  \begin{tabular}{c|cc|ccc}
    \toprule
                   & \multicolumn{2}{c|}{Laplace (interior)}                              & \multicolumn{3}{c}{Helmholtz (exterior)}                                                                                    \\
                   & Single-layer                           & Double-layer                           & Single-layer                             & Double-layer                             & CFIE                                       \\
    \midrule
    Consistency    & \cref{lem:consistency-laplace-SL}      & \cref{lem:consistency-laplace-DL}      & \cref{lem:consistency-helmholtz-SL}      & \cref{lem:consistency-helmholtz-DL}      & \cref{lem:consistency-helmholtz-CFIE}      \\
    Intrinsic norm & \cref{thm:intrinsic-norm-laplace-SL}   & \cref{thm:intrinsic-norm-laplace-DL}   & \cref{thm:intrinsic-norm-helmholtz-SL}   & \cref{thm:intrinsic-norm-helmholtz-DL}   & \cref{thm:intrinsic-norm-helmholtz-CFIE}   \\
    Stronger norm  & \cref{thm:stronger-norm-laplace-SL}    & \cref{thm:stronger-norm-laplace-DL}    & \cref{thm:stronger-norm-helmholtz-SL}    & \cref{thm:stronger-norm-helmholtz-DL}    & \cref{thm:stronger-norm-helmholtz-CFIE}    \\
    Weaker norms   & \cref{thm:weaker-norm-laplace-SL}      & \cref{thm:weaker-norm-laplace-DL}      & \cref{thm:weaker-norm-helmholtz-SL}      & \cref{thm:weaker-norm-helmholtz-DL}      & \cref{thm:weaker-norm-helmholtz-CFIE}      \\
    Pointwise      & \cref{thm:ptwise-eval-laplace-SL}      & \cref{thm:ptwise-eval-laplace-DL}      & \cref{thm:ptwise-eval-helmholtz-SL}      & \cref{thm:ptwise-eval-helmholtz-DL}      & \cref{thm:ptwise-eval-helmholtz-CFIE}      \\
    \bottomrule
  \end{tabular}
  \label{tab:summary-results}
\end{table}}

\section{Convergence rates for the Laplace equation}\label{sec:laplace}

We now apply the abstract, theoretical results of \cref{app:theory} to the Laplace equation.

\subsection{Single-layer potential}

We consider the following weak formulation of \cref{pb:laplace-SL}.

\begin{problem}[Weak formulation]\label{pb:weak-laplace-SL}
Find $p\in H^{-1/2}(\Gamma)$ such that
\begin{align*}
  b(p,q) = \langle f,q\rangle \quad \forall q\in H^{-1/2}(\Gamma),
\end{align*}
with $b:H^{-1/2}(\Gamma)\times H^{-1/2}(\Gamma)\to\R$ defined by
\begin{align*}
  b(p,q) = \langle S_0p,q\rangle = \frac{1}{4\pi}\int_{\Gamma}\int_{\Gamma}\frac{p(\bs{y})q(\bs{x})}{\vert\bs{x}-\bs{y}\vert}d\Gamma(\bs{y})d\Gamma(\bs{x}).
\end{align*}
We assume \( f \in H^{m+2}(\Gamma) \), so that \( p \in H^{m+1}(\Gamma) \), where \( m \geq 0 \) is the polynomial degree in \( V_h \).
\end{problem}

Since $b$ is coercive \cite[Thm.~1.1]{nedelec1973}, there is a unique solution to \cref{pb:weak-laplace-SL} via Lax--Milgram theorem (\cref{thm:lax-milgram}). In \cref{pb:weak-laplace-SL}, the expression $\langle f, q \rangle$ denotes the duality pairing between $f \in H^{1/2}(\Gamma)$ and $q \in H^{-1/2}(\Gamma)$. When $q\in L^2(\Gamma)$, this pairing is given by the $L^2(\Gamma)$-inner product
\[
  \langle f, q \rangle = (f, q) = \int_\Gamma f(\bs{x}) q(\bs{x}) \, d\Gamma(\bs{x}) \quad \forall f \in H^{1/2}(\Gamma), \ \forall q \in L^2(\Gamma).
\]
By density, this definition extends uniquely and continuously to all $q \in H^{-1/2}(\Gamma)$.

We consider the approximation spaces $V_h$ and $\hat{V}_h$ defined in \cref{subsec:FEM} for the single-layer potential.

The integration on $\Gamma_h$ instead of $\Gamma$ yields the following perturbed Galerkin formulation.

\begin{problem}[Perturbed Galerkin approximation problem]\label{pb:perturbed-laplace-SL}
Find $p_h\in V_h$ such that
\begin{align*}
  b_h(p_h,q_h) = (f_h,q_h)_h \quad \forall q_h\in V_h,
\end{align*}
with $b_h:H^{-1/2}(\Gamma_h)\times H^{-1/2}(\Gamma_h)\to\R$ defined by
\begin{align*}
  b_h(p,q) = \frac{1}{4\pi}\int_{\Gamma_h}\int_{\Gamma_h}\frac{p(\bs{y})q(\bs{x})}{\vert\bs{x}-\bs{y}\vert}d\Gamma_h(\bs{y})d\Gamma_h(\bs{x}).
\end{align*}
Equivalently, by changing variables with the pullback $\Psi_h^{-1}$, find $\hat{p}_h\in\hat{V}_h$ such that
\begin{align*}
  \hat{b}_h(\hat{p}_h,\hat{q}_h) = (\hat{f}_h J^{-1}_h,\hat{q}_h) \quad \forall \hat{q}_h\in \hat{V}_h,
\end{align*}
with $\hat{b}_h:H^{-1/2}(\Gamma)\times H^{-1/2}(\Gamma)\to\R$ defined by
\begin{align*}
  \hat{b}_h(\hat{p},\hat{q}) = \frac{1}{4\pi}\int_{\Gamma}\int_{\Gamma}\frac{\hat{p}(\bs{y})\hat{q}(\bs{x})J^{-1}_h(\bs{y})J^{-1}_h(\bs{x})}{\vert\Psi^{-1}_h(\bs{x})-\Psi^{-1}_h(\bs{y})\vert}d\Gamma(\bs{y})d\Gamma(\bs{x}).
\end{align*}
\end{problem}

We start by showing consistency.

\begin{lemma}[Consistency]\label{lem:consistency-laplace-SL}
  There exists \(h_0>0\) such that for all \(h\leq h_0\), the bilinear forms defined in \cref{pb:weak-laplace-SL} and \cref{pb:perturbed-laplace-SL} satisfy the consistency conditions
  \begin{align*}
    \vert b(\hat{p}_h,\hat{q}_h) - \hat{b}_h(\hat{p}_h,\hat{q}_h)\vert \leq c h^{\ell+1} \Vert\hat{p}_h\Vert_{0}\Vert\hat{q}_h\Vert_{0} \quad \forall\hat{p}_h,\hat{q}_h\in\hat{V}_h.
  \end{align*}
\end{lemma}

\begin{proof}
  We write
  \begin{align*}
      & \; b(\hat{p}_h,\hat{q}_h) - \hat{b}_h(\hat{p}_h,\hat{q}_h)                                                                                                                                                                \\
    = & \; \frac{1}{4\pi}\int_{\Gamma}\int_{\Gamma}\hat{p}_h(\bs{y})\hat{q}_h(\bs{x})\left[\frac{1}{\vert\bs{x}-\bs{y}\vert} - \frac{1}{\vert\Psi^{-1}_h(\bs{x})-\Psi^{-1}_h(\bs{y})\vert}\right]d\Gamma(\bs{y})d\Gamma(\bs{x})   \\
      & + \; \frac{1}{4\pi}\int_{\Gamma}\int_{\Gamma}\frac{\hat{p}_h(\bs{y})\hat{q}_h(\bs{x})}{\vert\Psi^{-1}_h(\bs{x})-\Psi^{-1}_h(\bs{y})\vert}\left[1 - J^{-1}_h(\bs{y})J^{-1}_h(\bs{x})\right]d\Gamma(\bs{y})d\Gamma(\bs{x}).
  \end{align*}
  This is now straightforward using the estimates from \cref{lem:geometry},
  \begin{align}\label{eq:consistency}
    \vert b(\hat{p}_h,\hat{q}_h) - \hat{b}_h(\hat{p}_h,\hat{q}_h)\vert \leq ch^{\ell+1}\frac{1}{4\pi}\int_{\Gamma}\int_{\Gamma}\frac{\vert\hat{p}_h(\bs{y})\vert\vert\hat{q}_h(\bs{x})\vert}{\vert\bs{x}-\bs{y}\vert}d\Gamma(\bs{y})d\Gamma(\bs{x}),
  \end{align}
  yielding the result via the continuity of $b$ in $L^2(\Gamma)\times L^2(\Gamma)$.
\end{proof}

Combining \cref{lem:consistency-laplace-SL} with inverse Sobolev inequalities (see \cref{lem:approximation-SL}),
we obtain a consistency estimate in the \(H^{-1/2}(\Gamma)\)-norm. Specifically, there exists \(h_0>0\) such that for all \(h\leq h_0\),
\begin{align}\label{eq:consistency-b}
  \vert b(\hat{p}_h,\hat{q}_h) - \hat{b}_h(\hat{p}_h,\hat{q}_h)\vert \leq c h^{\ell} \Vert\hat{p}_h\Vert_{-1/2}\Vert\hat{q}_h\Vert_{-1/2} \quad \forall\hat{p}_h,\hat{q}_h\in\hat{V}_h.
\end{align}
From this, we deduce the uniform coercivity of \( \hat{b}_h \) on \( H^{-1/2}(\Gamma) \times H^{-1/2}(\Gamma) \) for sufficiently small~\( h \), using \cref{rem:uniform-coercivity}. As pointed out in~\cite[Ex.~8.2.7]{sauter2011}, \cref{eq:consistency-b} cannot be obtained directly from \cref{eq:consistency} and the continuity of \( b \) on \( H^{-1/2}(\Gamma) \times H^{-1/2}(\Gamma) \), since in general, for \( p \in H^{-1/2}(\Gamma) \), $\Vert\vert p\vert\Vert_{-1/2}\neq\Vert p\Vert_{-1/2}$.

The uniform coercivity of $\hat{b}_h$ is the key ingredient for applying Strang's lemma (\cref{lem:strang-A}), which yields the well-posedness of \cref{pb:perturbed-laplace-SL} and the following estimate.

\begin{theorem}[Intrinsic norm]\label{thm:intrinsic-norm-laplace-SL}
  Let $p$ and $\hat{p}_h$ denote the solutions to \cref{pb:weak-laplace-SL} and \cref{pb:perturbed-laplace-SL} for sufficiently small $h$. Then
  \begin{align*}
    \Vert p - \hat{p}_h\Vert_{-1/2} \leq c\left[h^{m+3/2}\Vert f\Vert_{m+2} + h^{\ell+1/2}(\Vert f\Vert_{1} {\color{black}+ \lVert \nabla F\rVert_{L^{\infty}(\mathcal{N})}})\right].
  \end{align*}
\end{theorem}

\begin{proof}
  We apply Strang's lemma in the coercive case (\cref{lem:strang-A}) with $v_h=\hat{s}_hp$, where $\hat{s}_h$ denotes $L^2(\Gamma)$-orthogonal projector onto $\hat{V}_h$. This yields
  \begin{align*}
    \Vert p - \hat{p}_h\Vert_{-1/2} \leq c\hspace{-0.05cm}\left[\sup_{\hat{q}_h\in \hat{V}_h}\frac{\vert(f,\hat{q}_h) - (\hat{f}_hJ^{-1}_h,\hat{q}_h)\vert}{\Vert\hat{q}_h\Vert_{-1/2}} + \Vert p - \hat{s}_hp\Vert_{-1/2}  + \sup_{\hat{q}_h\in\hat{V}_h}\frac{\vert b(\hat{s}_hp,\hat{q}_h) - \hat{b}_h(\hat{s}_hp,\hat{q}_h)\vert}{\Vert\hat{q}_h\Vert_{-1/2}}\right].
  \end{align*}
  For the first {\color{black}term, we use \cref{lem:rhs_consistency} combined with an inverse Sobolev inequality for $\Vert \hat{q}_h\Vert_0$:
  \begin{align*}
    \lvert (f - \hat{f}_h J^{-1}_h,\hat{q}_h) \rvert \leq c h^{\ell+1/2} (\lVert f\rVert_0 + \lVert \nabla F\rVert_{L^{\infty}(\mathcal{N})}) \Vert\hat{q}_h\Vert_{-1/2}.
  \end{align*}}
  The bound on the second term follows from the approximation properties in \cref{lem:approximation-SL},
  \begin{align*}
    \Vert p - \hat{s}_hp\Vert_{-1/2} \leq ch^{m+3/2}\Vert p\Vert_{m+1} \leq ch^{m+3/2}\Vert f\Vert_{m+2}.
  \end{align*}

  Finally, for the third term, we use \cref{lem:consistency-laplace-SL} and an inverse Sobolev inequality,
  \begin{align*}
    \vert b(\hat{s}_hp,\hat{q}_h) - \hat{b}_h(\hat{s}_hp,\hat{q}_h)\vert \leq ch^{\ell+1/2}\Vert\hat{s}_hp\Vert_{0}\Vert\hat{q}_h\Vert_{-1/2},
  \end{align*}
  and \cref{lem:approximation-SL} to bound $\Vert\hat{s}_hp\Vert_{0}\leq\Vert p\Vert_{0}\leq c\Vert f\Vert_1$.
\end{proof}

Note that since $V_h$ consists of piecewise polynomials of degree at most $m$, the approximation error depends on the $(m+1)$-th derivatives of $p$ and, in general, cannot decay faster than $h^{m+3/2}$ in the $H^{-1/2}(\Gamma)$-norm, regardless of the smoothness of $p$.

We can also prove a theorem in the $L^2(\Gamma)$-norm.

\begin{theorem}[Stronger norm]\label{thm:stronger-norm-laplace-SL}
  Let $p$ and $\hat{p}_h$ denote the solutions to \cref{pb:weak-laplace-SL} and \cref{pb:perturbed-laplace-SL} for sufficiently small $h$. Then
  \begin{align*}
    \Vert p - \hat{p}_h\Vert_{0} \leq c\left[h^{m+1}\Vert f\Vert_{m+2} + h^{\ell}(\Vert f\Vert_{1} {\color{black}+ \lVert \nabla F\rVert_{L^{\infty}(\mathcal{N})}})\right].
  \end{align*}
\end{theorem}

\begin{proof}
  We write
  \begin{align*}
    \Vert p - \hat{p}_h\Vert_{0} \leq \Vert p - \hat{s}_hp\Vert_{0} + \Vert\hat{s}_hp - \hat{p}_h\Vert_{0}.
  \end{align*}
  For the first term, from \cref{lem:approximation-SL}, we have $\Vert p - \hat{s}_hp\Vert_{0} \leq c h^{m+1}\Vert p\Vert_{m+1}$. For the second term,
  \begin{align*}
    \Vert\hat{s}_hp - \hat{p}_h\Vert_{0} & \leq ch^{-1/2}\left[\Vert\hat{s}_hp - p\Vert_{-1/2} + \Vert p - \hat{p}_h\Vert_{-1/2}\right] \leq ch^{-1/2}\left[h^{m+3/2}\Vert p\Vert_{m+1} + \Vert p - \hat{p}_h\Vert_{-1/2}\right],
  \end{align*}
  again by \cref{lem:approximation-SL}. We then use \cref{thm:intrinsic-norm-laplace-SL} and $\Vert p\Vert_{m+1}\leq c\Vert f\Vert_{m+2}$ to conclude.
\end{proof}

The main tool to prove pointwise estimates such as those in \cref{eq:pointwise_single} is to first prove results in weaker Sobolev norms. These results rely on duality arguments such as the Aubin--Nitsche lemma, and seem to go back to Hsiao and Wendland \cite{hsiao1981}; see also \cite{sauter2011, schatz1990, sloan1992}.

\begin{theorem}[Weaker norms]\label{thm:weaker-norm-laplace-SL}
  Let $p$ and $\hat{p}_h$ denote the solutions to \cref{pb:weak-laplace-SL} and \cref{pb:perturbed-laplace-SL} for sufficiently small $h$. Then
  \begin{align*}
    \Vert p - \hat{p}_h\Vert_{-m-2} \leq c\left[h^{2m+3}\Vert f\Vert_{m+2} + h^{\ell+1}(\Vert f\Vert_{1} {\color{black}+ \lVert \nabla F\rVert_{L^{\infty}(\mathcal{N})}})\right].
  \end{align*}
\end{theorem}

\begin{proof}
  We have
  \begin{align*}
    \Vert p - \hat{p}_h\Vert_{-m-2} = \sup_{g\in H^{m+2}(\Gamma)}\frac{\vert\langle g,p-\hat{p}_h\rangle\vert}{\Vert g\Vert_{m+2}} = \sup_{g\in H^{m+2}(\Gamma)}\frac{\vert b(p-\hat{p}_h,q)\vert}{\Vert g\Vert_{m+2}},
  \end{align*}
  where $q\in H^{-1/2}(\Gamma)$ solves the following dual problem,
  \begin{align*}
    b(p,q) = \langle g,p\rangle \quad \forall p\in H^{-1/2}(\Gamma),
  \end{align*}
  which is the same problem as \cref{pb:weak-laplace-SL} by symmetry of $b$,
  \begin{align*}
    b(p,q) = \langle S_0p,q\rangle = \langle S_0q, p\rangle = b(q,p).
  \end{align*}
  Since $g\in H^{m+2}(\Gamma)$ yields $q\in H^{m+1}(\Gamma)$ and $\Vert q\Vert_{m+1}\leq c\Vert g\Vert_{m+2}$, we arrive at
  \begin{align*}
    \Vert p - \hat{p}_h\Vert_{-m-2} \leq c\sup_{q\in H^{m+1}(\Gamma)}\frac{\vert b(p-\hat{p}_h,q)\vert}{\Vert q\Vert_{m+1}}.
  \end{align*}
  Now, write $b(p-\hat{p}_h,q) = b(p-\hat{p}_h,q-\hat{s}_hq) + b(p-\hat{p}_h,\hat{s}_hq)$. The first term can be bounded as
  \begin{align*}
    \vert b(p-\hat{p}_h,q-\hat{s}_hq)\vert \leq c\Vert p-\hat{p}_h\Vert_{-1/2}\Vert q-\hat{s}_hq\Vert_{-1/2}.
  \end{align*}
  (Recall $\hat{s}_h$ denotes the $L^2(\Gamma)$-projector onto $\hat{V}_h$.) We use \cref{thm:intrinsic-norm-laplace-SL} and \cref{lem:approximation-SL} to obtain
  \begin{align*}
    \vert b(p-\hat{p}_h,q-\hat{s}_hq)\vert \leq c\left[h^{2m+3}\Vert f\Vert_{m+2} + h^{\ell+m+2}(\Vert f\Vert_{1} {\color{black}+ \lVert \nabla F  \rVert_{L^{\infty}(\mathcal{N})}})\right]\Vert q\Vert_{m+1}.
  \end{align*}

  For the second term, we write
  \begin{align*}
    b(p-\hat{p}_h,\hat{s}_hq) = & \, b(p,\hat{s}_hq) - \hat{b}_h(\hat{p}_h,\hat{s}_hq) + \hat{b}_h(\hat{p}_h,\hat{s}_hq) - b(\hat{p}_h,\hat{s}_hq).
  \end{align*}
  For the first difference, we observe that $b(p,\hat{s}_hq) = (f,\hat{s}_hq)$ and $\hat{b}_h(\hat{p}_h,\hat{s}_hq) = (\hat{f}_hJ^{-1}_h,\hat{s}_hq)$, hence
  \begin{align*}
    b(p,\hat{s}_hq) - \hat{b}_h(\hat{p}_h,\hat{s}_hq) = (f - \hat{f}_hJ^{-1}_h,\hat{s}_hq).
  \end{align*}
  {\color{black}We can bound this term using \cref{lem:rhs_consistency} and \cref{lem:approximation-SL},
  \begin{align*}
    \vert(f - \hat{f}_hJ^{-1}_h,\hat{s}_hq)\vert \leq ch^{\ell+1}(\Vert f\Vert_1 + \lVert \nabla F\rVert_{L^{\infty}(\mathcal{N})})\Vert q\Vert_{m+1}.
  \end{align*}}
  For the second difference, we write
  \begin{align*}
    \vert\hat{b}_h(\hat{p}_h,\hat{s}_hq) - b(\hat{p}_h,\hat{s}_hq)\vert \leq ch^{\ell+1}\Vert\hat{p}_h\Vert_{0}\Vert\hat{s}_hq\Vert_{0}.
  \end{align*}
  Since $\Vert\hat{s}_hq\Vert_{0}\leq \Vert q\Vert_{0}\leq\Vert q\Vert_{m+1}$ and
  \begin{align*}
    \Vert\hat{p}_h\Vert_{0} \leq \Vert p\Vert_{0}  + \Vert p - \hat{p}_h\Vert_{0} \leq c\left[\Vert p\Vert_{0} + h^{m+1}\Vert f\Vert_{m+2} + h^{\ell}(\Vert f\Vert_{1} {\color{black}+ \lVert \nabla F  \rVert_{L^{\infty}(\mathcal{N})}})\right],
  \end{align*}
  via \cref{thm:stronger-norm-laplace-SL}, \textcolor{black}{we bound $\vert\hat{b}_h(\hat{p}_h,\hat{s}_hq)-b(\hat{p}_h,\hat{s}_hq)\vert$ by}
  \begin{align*}
    c\left[h^{\ell+1}\Vert f\Vert_{1} + h^{\ell+m+2}\Vert f\Vert_{m+2} + h^{2\ell+1}(\Vert f\Vert_{1} {\color{black}+ \lVert \nabla F  \rVert_{L^{\infty}(\mathcal{N})}})\right]\Vert q\Vert_{m+1}.
  \end{align*}
  {\color{black} Collecting the terms and keeping only the leading-order ones gives the result.}
\end{proof}

Once results in weaker norms are established, the following pointwise estimates follow directly. The key is to bound integrals of differences by products of \( H^{-m-2}(\Gamma) \)- and \( H^{m+2}(\Gamma) \)-norms.

\begin{theorem}[Pointwise evaluation]\label{thm:ptwise-eval-laplace-SL}
  Let $p$ and $\hat{p}_h$ denote the solutions to \cref{pb:weak-laplace-SL} and \cref{pb:perturbed-laplace-SL} for sufficiently small $h$. Then for all $\bs{x}\in\Omega$
  \begin{align*}
     \vert u(\bs{x}) - u_h(\bs{x})\vert \leq c_{\bs{x}}\left[h^{2m+3}\Vert f\Vert_{m+2} + h^{\ell+1}(\Vert f\Vert_{1} {\color{black}+ \lVert \nabla F\rVert_{L^{\infty}(\mathcal{N})}} )\right],
  \end{align*}
  with $c_{\bs{x}}\to\infty$ as $\bs{x}\to\Gamma$.
\end{theorem}

\begin{proof}
  Let $\bs{x}\in\Omega$. We write
  \begin{align*}
    u(\bs{x}) - u_h(\bs{x}) = \frac{1}{4\pi}\int_\Gamma\left[\frac{p(\bs{y})}{\vert\bs{x}-\bs{y}\vert}-\frac{\hat{p}_h(\bs{y})J^{-1}_h(\bs{y})}{\vert\bs{x}-\Psi^{-1}_h(\bs{y})\vert}\right]d\Gamma(\bs{y}),
  \end{align*}
  which we split into
  \begin{align*}
    u(\bs{x}) - u_h(\bs{x}) = & \; \frac{1}{4\pi}\int_\Gamma\left[\frac{p(\bs{y})}{\vert\bs{x}-\bs{y}\vert}-\frac{\hat{p}_h(\bs{y})}{\vert\bs{x}-\bs{y}\vert}\right]d\Gamma(\bs{y}) + \frac{1}{4\pi}\int_\Gamma\hat{p}_h(\bs{y})\left[\frac{1}{\vert\bs{x}-\bs{y}\vert}-\frac{1}{\vert\bs{x} -\Psi^{-1}_h(\bs{y})\vert}\right]d\Gamma(\bs{y}) \\
                              & \; + \frac{1}{4\pi}\int_\Gamma\hat{p}_h(\bs{y})\frac{1 - J^{-1}_h(\bs{y})}{\vert\bs{x} -\Psi^{-1}_h(\bs{y})\vert}d\Gamma(\bs{y}).
  \end{align*}
  {\color{black}Since $\bs{x}\notin\Gamma$, the function $\vert\bs{x}-\cdot\vert^{-1}$ is smooth on $\Gamma$. By duality between $H^{-m-2}(\Gamma)$ and $H^{m+2}(\Gamma)$, we bound the first term by
  \begin{align*}
    c\Vert p-\hat{p}_h\Vert_{-m-2}\Vert\vert\bs{x}-\cdot\vert^{-1}\Vert_{m+2}\leq c_{\bs{x}}\Vert p-\hat{p}_h\Vert_{-m-2},
  \end{align*}
  and conclude with \cref{thm:weaker-norm-laplace-SL}. For the second and third terms, \cref{lem:geometry} and the Cauchy--Schwarz inequality bound both terms by $c_{\bs{x}}h^{\ell+1}\Vert\hat{p}_h\Vert_{0}$. By \cref{thm:stronger-norm-laplace-SL},
  \begin{align*}
    \Vert\hat{p}_h\Vert_{0} \leq \Vert p\Vert_{0}  + \Vert p - \hat{p}_h\Vert_{0} \leq c\left[\Vert p\Vert_{0} + h^{m+1}\Vert f\Vert_{m+2} + h^{\ell}(\Vert f\Vert_{1} {\color{black}+ \lVert \nabla F  \rVert_{L^{\infty}(\mathcal{N})}})\right],
  \end{align*}
  and $\Vert p\Vert_0\leq c\Vert f\Vert_1$ as in the proof of \cref{thm:intrinsic-norm-laplace-SL}. Finally, $c_{\bs{x}}\to\infty$ as $\bs{x}\to\Gamma$.}
\end{proof}

To conclude this section, let us take a step back and compare our results with those in \cref{tab:results}. {\color{black}In \cite{sauter2011}, the effect of the boundary approximation is analyzed by combining the abstract perturbation results of \cite[Sec.~4.2]{sauter2011} with the geometric consistency estimates in \cite[Tab.~8.2]{sauter2011} evaluated in \cite[Cor.~8.2.9]{sauter2011}.} For the intrinsic norm, \cref{thm:intrinsic-norm-laplace-SL} matches the results of \cite[Thm.~2.1]{nedelec1976} and {\color{black}those obtained from} \cite[Thm.~4.2.11]{sauter2011} {\color{black}and \cite[Cor.~8.2.9]{sauter2011}}. However, in the case of weaker norms (\cref{thm:weaker-norm-laplace-SL}), the estimate {\color{black}obtained from} \cite[Thm.~4.2.19]{sauter2011} {\color{black}and \cite[Cor.~8.2.9]{sauter2011}} loses a factor of \(\sqrt{h}\) in the geometric error. This loss comes from relying on consistency in \(L^2(\Gamma) \times H^{-1/2}(\Gamma)\) rather than \(L^2(\Gamma) \times L^2(\Gamma)\), and from applying an inverse Sobolev inequality (see \cite[Cor.~8.2.6]{sauter2011}). In contrast, N\'{e}d\'{e}lec's work (\cite[Thm.~2.2]{nedelec1976}) provides a sharp geometric error in weaker norms, but the estimate loses a factor \(h^{m+1}\) because it does not fully exploit the regularity of the problem.

\subsection{Double-layer potential}\label{subsec:laplace-DL}

We consider the following weak formulation of \cref{pb:laplace-DL}.

\begin{problem}[Weak formulation]\label{pb:weak-laplace-DL}
Find $p\in L^{2}(\Gamma)$ such that
\begin{align*}
  b(p,q) = (f, q) \quad \forall q\in L^{2}(\Gamma),
\end{align*}
with $b:L^{2}(\Gamma)\times L^{2}(\Gamma)\to\R$ defined by $b(p,q) = (p,q)/2 - (D_0p,q)$, that is,
\begin{align*}
  b(p,q) = \frac{1}{2} \int_{\Gamma} p(\bs{x})q(\bs{x}) d\Gamma(\bs{x}) -\frac{1}{4\pi}\int_\Gamma \int_{\Gamma}\frac{\partial }{\partial \bs{n}(\bs{y})} \left(\frac{1}{\vert\bs{x}-\bs{y}\vert}\right) p(\bs{y})q(\bs{x}) d\Gamma(\bs{y})d\Gamma(\bs{x}).
\end{align*}
We assume \( f \in H^{m+1}(\Gamma) \), so that \( p \in H^{m+1}(\Gamma) \), where \( m \geq 0 \) is the polynomial degree in \( V_h \).
\end{problem}

Since $b$ is injective and can be rewritten as $b(p, q) = a(p, q) + t(p, q)$ with
\[
a(p, q) = \frac{1}{2}(p, q) \quad \text{(coercive in $L^2(\Gamma)$)}, \qquad
t(p, q) = -(D_0 p, q) \quad \text{(compact in $L^2(\Gamma)$)},
\]
there is a unique solution to \cref{pb:weak-laplace-DL} by Fredholm's alternative (\cref{thm:fredholm}). The compactness of \(D_0\) in $L^2(\Gamma)$ is guaranteed by the smoothness of $\Gamma$; see \cite{chandler2012} and the references therein.

Before continuing, we note that
\begin{align}\label{eq:laplace_bilinear_form_rewritten}
  b(p,q) = (p,q)+\frac{1}{4\pi}\int_\Gamma\int_\Gamma (p(\bs{x}) - p(\bs{y})) q(\bs{x})\frac{(\bs{x}-\bs{y})\cdot \bs{n}(\bs{y})}{\vert\bs{x}-\bs{y}\vert^3} d\Gamma(\bs{y})d\Gamma(\bs{x}).
\end{align}
It follows from~\cite[eq.~(3.34)]{nedelec1977}:
\begin{align*}
  \frac{1}{2} = -\frac{1}{4\pi}\int_{\Gamma}\frac{\partial}{\partial \bs{n}(\bs{y})} \left(\frac{1}{\vert\bs{x}-\bs{y}\vert}\right)d\Gamma(\bs{y}) \quad \forall\bs{x}\in\Gamma.
\end{align*}

We consider the approximation spaces $V_h$ and $\hat{V}_h$ defined in \cref{subsec:FEM} for the double-layer potential. We denote by \(\bs{n}_h(\bs{y})\) the unit normal vector pointing outward from the curved element at the point \(\bs{y}\) and by \(\hat{\bs{n}}_h = \bs{n}_h \circ \Psi_h^{-1}\) the lifted normal. \textcolor{black}{Earlier analyses~\cite{giroire1978b, nedelec1977} use instead the interpolated normal \(\bs{\nu}_h\), obtained on each element by interpolating the exact normal \(\bs{n}\) at the nodes; we only use \(\bs{n}_h\), which is standard in boundary element codes.}

\begin{problem}[Perturbed Galerkin approximation problem]\label{pb:perturbed-laplace-DL}
Find $\hat{p}_h\in\hat{V}_h$ such that
\begin{align*}
  \hat{b}_h(\hat{p}_h,\hat{q}_h) = (\hat{f}_hJ^{-1}_h,\hat{q}_h) \quad \forall \hat{q}_h\in \hat{V}_h,
\end{align*}
with $\hat{b}_h:L^{2}(\Gamma)\times L^{2}(\Gamma)\to\R$ defined by
\begin{align*}
  \hat{b}_h(\hat{p},\hat{q}) = & \int_{\Gamma} \hat{p}(\bs{x})\hat{q}(\bs{x})J^{-1}_h(\bs{x})d\Gamma(\bs{x})                                                                                                                                                                                                                      \\
                               & + \frac{1}{4\pi}\int_{\Gamma} \int_{\Gamma} (\hat{p}(\bs{x}) - \hat{p}(\bs{y})) \hat{q}(\bs{x})\frac{(\Psi^{-1}_h(\bs{x})-\Psi^{-1}_h(\bs{y}))\cdot \hat{\bs{n}}_h(\bs{y})}{\vert\Psi^{-1}_h(\bs{x})-\Psi^{-1}_h(\bs{y})\vert^3}J^{-1}_h(\bs{y})J^{-1}_h(\bs{x}) d\Gamma(\bs{y})d\Gamma(\bs{x}).
\end{align*}
\end{problem}

\begin{lemma}[Consistency]\label{lem:consistency-laplace-DL}
  There exists \(h_0>0\) such that for all \(h\leq h_0\), the bilinear forms defined in \cref{pb:weak-laplace-DL} and \cref{pb:perturbed-laplace-DL} satisfy the consistency conditions
  \begin{align*}
    \vert b(\hat{p}_h,\hat{q}_h) - \hat{b}_h(\hat{p}_h,\hat{q}_h)\vert \leq c_\epsilon h^{\ell} \Vert\hat{p}_h\Vert_{\epsilon}\Vert\hat{q}_h\Vert_{0} \quad {\color{black}\forall \epsilon\in(0,1/2),} \; \forall\hat{p}_h,\hat{q}_h\in\hat{V}_h,
  \end{align*}
  with $c_\epsilon\to\infty$ as $\epsilon\to0$.
\end{lemma}

\begin{proof}
  Let \textcolor{black}{$0<\epsilon<1/2$}. We write
  \begin{align*}
      & \; b(\hat{p}_h,\hat{q}_h) - \hat{b}_h(\hat{p}_h,\hat{q}_h)                                                                                                                                                                                                       \\
    = & \; \int_{\Gamma}\hat{p}_h(\bs{x})\hat{q}_h(\bs{x})(1-J^{-1}_h(\bs{x}))d\Gamma(\bs{x})                                                                                                                                                                            \\
      & + \; \frac{1}{4\pi}\int_{\Gamma}\int_{\Gamma}(\hat{p}_h(\bs{x})-\hat{p}_h(\bs{y}))\hat{q}_h(\bs{x})                                                                                                                                                              \\
      & \times \; \left[\frac{(\bs{x}-\bs{y})\cdot \bs{n}(\bs{y})}{\vert\bs{x}-\bs{y}\vert^3} - \frac{(\Psi^{-1}_h(\bs{x})-\Psi^{-1}_h(\bs{y}))\cdot\hat{\bs{n}}_h(\bs{y})}{\vert\Psi^{-1}_h(\bs{x})-\Psi^{-1}_h(\bs{y})\vert^3} J^{-1}_h(\bs{y})J^{-1}_h(\bs{x})\right]
    d\Gamma(\bs{y})d\Gamma(\bs{x}).
  \end{align*}
  The first term is immediately bounded by $ch^{\ell+1}\Vert\hat{p}_h\Vert_{0}\Vert\hat{q}_h\Vert_{0}\leq ch^{\ell+1}\Vert\hat{p}_h\Vert_{\epsilon}\Vert\hat{q}_h\Vert_{0}$. For the second term we utilize the following decomposition,
  \begin{align*}
      & \; \frac{(\bs{x}-\bs{y})\cdot \bs{n}(\bs{y})}{\vert\bs{x}-\bs{y}\vert^3} - \frac{(\Psi^{-1}_h(\bs{x})-\Psi^{-1}_h(\bs{y}))\cdot\hat{\bs{n}}_h(\bs{y})}{\vert\Psi^{-1}_h(\bs{x})-\Psi^{-1}_h(\bs{y})\vert^3} J^{-1}_h(\bs{y})J^{-1}_h(\bs{x}) \\
    = & \; (\bs{x}-\bs{y})\cdot\bs{n}(\bs{y})\left[\frac{1}{\vert\bs{x}-\bs{y}\vert^3} - \frac{1}{\vert\Psi^{-1}_h(\bs{x})-\Psi^{-1}_h(\bs{y})\vert^3}\right]
    + \frac{(\bs{x}-\bs{y}) - (\Psi^{-1}_h(\bs{x})-\Psi^{-1}_h(\bs{y}))}{\vert\Psi^{-1}_h(\bs{x})-\Psi^{-1}_h(\bs{y})\vert^3}\cdot\bs{n}(\bs{y})                                                                                                     \\
      & + \; \frac{\Psi^{-1}_h(\bs{x})-\Psi^{-1}_h(\bs{y})}{\vert\Psi^{-1}_h(\bs{x})-\Psi^{-1}_h(\bs{y})\vert^3}\cdot\left[\bs{n}(\bs{y}) - \hat{\bs{n}}_h(\bs{y})\right]
    + \frac{(\Psi^{-1}_h(\bs{x})-\Psi^{-1}_h(\bs{y}))\cdot\hat{\bs{n}}_h(\bs{y})}{\vert\Psi^{-1}_h(\bs{x})-\Psi^{-1}_h(\bs{y})\vert^3}\left[1 - J^{-1}_h(\bs{y})J^{-1}_h(\bs{x})\right].
  \end{align*}
  The first and last terms are bounded by \( ch^{\ell+1} \vert \bs{x} - \bs{y} \vert^{-2} \), while the second and third are bounded by \( ch^{\ell} \vert \bs{x} - \bs{y} \vert^{-2} \); see \cref{lem:geometry}. This gives
  \begin{align*}
    ch^{\ell} \int_{\Gamma} \int_{\Gamma} \frac{ \vert \hat{p}_h(\bs{x}) - \hat{p}_h(\bs{y}) \vert \, \vert \hat{q}_h(\bs{x}) \vert }{ \vert \bs{x} - \bs{y} \vert^2 } \, d\Gamma(\bs{y}) \, d\Gamma(\bs{x}).
  \end{align*}
  Using the Cauchy--Schwarz inequality, we write
  \begin{align*}
    \int_{\Gamma}\int_{\Gamma}\frac{\vert\hat{p}_h(\bs{x})-\hat{p}_h(\bs{y})\vert\vert\hat{q}_h(\bs{x})\vert}{\vert\bs{x}-\bs{y}\vert^2}d\Gamma(\bs{y})d\Gamma(\bs{x}) \leq & \; \left(\int_{\Gamma}\int_{\Gamma}\frac{\vert\hat{p}_h(\bs{x})-\hat{p}_h(\bs{y})\vert^2}{\vert\bs{x}-\bs{y}\vert^{2+2\epsilon}}d\Gamma(\bs{y})d\Gamma(\bs{x})\right)^{1/2} \\
                                                                                                                                                                            & \; \times \left(\int_{\Gamma}\int_{\Gamma}\frac{\vert\hat{q}_h(\bs{x})\vert^2}{\vert\bs{x}-\bs{y}\vert^{2-2\epsilon}}d\Gamma(\bs{y})d\Gamma(\bs{x})\right)^{1/2}.
  \end{align*}
  The first term defines a Sobolev--Slobodeckij semi-norm in $H^{\epsilon}(\Gamma)$, which can be bounded by the $H^{\epsilon}(\Gamma)$-norm \cite[Rem.~10.5]{lions1972}. We bound the second one by $c_\epsilon\Vert\hat{q}_h\Vert_0$, {\color{black}with
 \begin{align*}
   c_\epsilon = \sup_{\bs{x}\in\Gamma}\left(\int_\Gamma\frac{d\Gamma(\bs{y})}{\vert\bs{x}-\bs{y}\vert^{2-2\epsilon}}\right)^{1/2},
 \end{align*}}
 and $c_\epsilon\to\infty$ as $\epsilon\to0$.
\end{proof}

Using an inverse Sobolev inequality, we obtain a consistency estimate in $L^2(\Gamma) \times L^2(\Gamma)$ (with factor $h^{\ell-\epsilon}$) and deduce that $\hat{b}_h$ satisfies the discrete inf-sup conditions uniformly in $L^2(\Gamma) \times L^2(\Gamma)$ for sufficiently small $h$ (see \cref{rem:uniform-infsup}). We can then apply Strang's lemma (\cref{lem:strang-B}).

\begin{theorem}[Intrinsic norm]\label{thm:intrinsic-norm-laplace-DL}
  Let $p$ and $\hat{p}_h$ denote the solutions to \cref{pb:weak-laplace-DL} and \cref{pb:perturbed-laplace-DL} for sufficiently small $h$. Then
  \begin{align*}
    \Vert p - \hat{p}_h\Vert_{0} \leq c\big[h^{m+1}\Vert f\Vert_{m+1} + {\color{black}h^{\ell}\Vert f\Vert_{1} + h^{\ell+1}(\Vert f\Vert_{0} + \lVert \nabla F \rVert_{L^{\infty}(\mathcal{N})})}\big].
  \end{align*}
\end{theorem}

\begin{proof}
  We apply Strang's lemma with $v_h=\hat{s}_hp$, where $\hat{s}_h$ denotes the $L^2(\Gamma)$-orthogonal projector onto $\hat{V}_h$, which yields
  \begin{align*}
    \Vert p - \hat{p}_h\Vert_{0} \leq c\left[\sup_{\hat{q}_h\in \hat{V}_h}\frac{\vert(f,\hat{q}_h) - (\hat{f}_hJ^{-1}_h,\hat{q}_h)\vert}{\Vert\hat{q}_h\Vert_{0}} + \Vert p - \hat{s}_hp\Vert_{0} + \sup_{\hat{q}_h\in\hat{V}_h}\frac{\vert b(\hat{s}_hp,\hat{q}_h) - \hat{b}_h(\hat{s}_hp,\hat{q}_h)\vert}{\Vert\hat{q}_h\Vert_{0}}\right].
  \end{align*}
  {\color{black}We can bound the first term using \cref{lem:rhs_consistency},
  \begin{align*}
    \vert(f - \hat{f}_hJ^{-1}_h,\hat{q}_h)\vert \leq ch^{\ell+1}(\Vert f\Vert_0  + \lVert \nabla F \rVert_{L^{\infty}(\mathcal{N})})\Vert\hat{q}_h\Vert_0.
  \end{align*}}

  The bound on the second term follows from \cref{lem:approximation-DL},
  \begin{align*}
    \Vert p - \hat{s}_hp\Vert_{0} \leq ch^{m+1}\Vert p\Vert_{m+1} \leq ch^{m+1}\Vert f\Vert_{m+1}.
  \end{align*}

  Finally, for the third term, we use \cref{lem:consistency-laplace-DL} with $h^\ell$,
  \begin{align*}
    \vert b(\hat{s}_hp,\hat{q}_h) - \hat{b}_h(\hat{s}_hp,\hat{q}_h)\vert \leq c_\epsilon h^{\ell}\Vert\hat{s}_hp\Vert_{\epsilon}\Vert\hat{q}_h\Vert_{0},
  \end{align*}
  for some \textcolor{black}{$0<\epsilon<1/2$}, and \cref{lem:approximation-DL} to bound $\Vert\hat{s}_hp\Vert_{\epsilon}\leq c\Vert p\Vert_{1}\leq c\Vert f\Vert_1$.
\end{proof}

{\color{black}
\begin{theorem}[Stronger norm]\label{thm:stronger-norm-laplace-DL}
  Let $p$ and $\hat{p}_h$ denote the solutions to \cref{pb:weak-laplace-DL} and \cref{pb:perturbed-laplace-DL} for sufficiently small $h$. Then
  \begin{align*}
    \Vert p - \hat{p}_h\Vert_{\epsilon} \leq c\left[h^{m+1-\epsilon}\Vert f\Vert_{m+1} + h^{\ell-\epsilon}\Vert f\Vert_{1} + h^{\ell+1-\epsilon}( \Vert f\Vert_{0}+ \lVert \nabla F \rVert_{L^{\infty}(\mathcal{N})})\right] \qquad \forall\epsilon\in(0,1/2).
  \end{align*}
\end{theorem}

\begin{proof}
  Let $\epsilon\in(0,1/2)$. We write
  \begin{align*}
    \Vert p - \hat{p}_h\Vert_{\epsilon} \leq \Vert p - \hat{s}_hp\Vert_{\epsilon} + \Vert\hat{s}_hp - \hat{p}_h\Vert_{\epsilon}.
  \end{align*}
  For the first term, from \cref{lem:approximation-DL}, we have that
  \begin{align*}
    \Vert p - \hat{s}_hp\Vert_{\epsilon} \leq c h^{m+1-\epsilon}\Vert p\Vert_{m+1} \leq c h^{m+1-\epsilon}\Vert f\Vert_{m+1}.
  \end{align*}

  For the second term, we also utilize \cref{lem:approximation-DL},
  \begin{align*}
    \Vert\hat{s}_hp - \hat{p}_h\Vert_{\epsilon} \leq ch^{-\epsilon}\left[\Vert\hat{s}_hp - p\Vert_{0} + \Vert p - \hat{p}_h\Vert_{0}\right] \leq ch^{-\epsilon}\left[h^{m+1}\Vert p\Vert_{m+1} + \Vert p - \hat{p}_h\Vert_{0}\right],
  \end{align*}
  and we use \cref{thm:intrinsic-norm-laplace-DL} and $\Vert p\Vert_{m+1}\leq c\Vert f\Vert_{m+1}$ to conclude.
\end{proof}
}

\begin{theorem}[Weaker norms]\label{thm:weaker-norm-laplace-DL}
  Let $p$ and $\hat{p}_h$ denote the solutions to \cref{pb:weak-laplace-DL} and \cref{pb:perturbed-laplace-DL} for sufficiently small $h$. Then
  \begin{align*}
    \Vert p - \hat{p}_h\Vert_{-m-1} \leq c\left[h^{2m+2}\Vert f\Vert_{m+1} + {\color{black}h^{\ell}\Vert f\Vert_{1} +  h^{\ell+1} (\Vert f \Vert_0 + \lVert \nabla F\rVert_{L^{\infty}(\mathcal{N})})}\right].
  \end{align*}
\end{theorem}

\begin{proof}
  We have
  \begin{align*}
    \Vert p - \hat{p}_h\Vert_{-m-1} = \sup_{g\in H^{m+1}(\Gamma)}\frac{\vert (g,p-\hat{p}_h)\vert}{\Vert g\Vert_{m+1}} = \sup_{g\in H^{m+1}(\Gamma)}\frac{\vert b(p-\hat{p}_h,q)\vert}{\Vert g\Vert_{m+1}},
  \end{align*}
  where $q \in L^2(\Gamma)$ solves the following dual problem,
  \begin{align*}
    b(p,q) = (g,p) \quad \forall p \in L^2(\Gamma).
  \end{align*}
  We note that
  \begin{align*}
    b(p,q) = (p,q)/2 - (D_0p,q) = (q,p)/2 - (D_0^*q,p) := b^*(q,p),
  \end{align*}
  where the dual operator $D_0^*: L^2(\Gamma) \to L^2(\Gamma)$ is defined by
  \begin{align*}
    (D_0^*q)(\bs{x}) = \frac{1}{4\pi}\int_\Gamma \frac{\partial}{\partial \bs{n}(\bs{x})} \left(\frac{1}{\vert \bs{x} - \bs{y} \vert}\right) q(\bs{y}) d\Gamma(\bs{y}).
  \end{align*}
  The dual problem is associated with the Laplace Neumann exterior problem \cite{mclean2000}. In particular, $g\in H^{m+1}(\Gamma)$ implies that $q \in H^{m+1}(\Gamma)$ with $\Vert q \Vert_{m+1} \leq c \Vert g \Vert_{m+1}$, which yields
  \begin{align*}
    \Vert p - \hat{p}_h \Vert_{-m-1} \leq c \sup_{q \in H^{m+1}(\Gamma)} \frac{\vert b(p - \hat{p}_h, q) \vert}{\Vert q \Vert_{m+1}}.
  \end{align*}
  Now, write $b(p - \hat{p}_h, q) = b(p - \hat{p}_h, q - \hat{s}_h q) + b(p - \hat{p}_h, \hat{s}_h q)$. The first term can be bounded as
  \begin{align*}
    \vert b(p - \hat{p}_h, q - \hat{s}_h q) \vert \leq c \Vert p - \hat{p}_h \Vert_{0} \Vert q - \hat{s}_h q \Vert_{0}.
  \end{align*}
  (Recall $\hat{s}_h$ denotes the $L^2(\Gamma)$-projector onto $\hat{V}_h$.) We use \cref{thm:intrinsic-norm-laplace-DL} and \cref{lem:approximation-DL} to get
  \begin{align*}
    \vert b(p - \hat{p}_h, q - \hat{s}_h q) \vert \leq c \left[h^{2m+2} \Vert f \Vert_{m+1} + {\color{black} h^{\ell+m+1} \Vert f \Vert_{1} + h^{\ell+m+2} (\Vert f \Vert_{0} + \lVert \nabla F \rVert_{L^{\infty}(\mathcal{N})}}) \right] \Vert q \Vert_{m+1}.
  \end{align*}

  For the second term, we write
  \begin{align*}
    b(p - \hat{p}_h, \hat{s}_h q) = b(p, \hat{s}_h q) - \hat{b}_h(\hat{p}_h, \hat{s}_h q) + \hat{b}_h(\hat{p}_h, \hat{s}_h q) - b(\hat{p}_h, \hat{s}_h q).
  \end{align*}
  For the first difference, we observe that $b(p, \hat{s}_h q) = (f, \hat{s}_h q)$ and
  $\hat{b}_h(\hat{p}_h, \hat{s}_h q) = (\hat{f}_h J^{-1}_h, \hat{s}_h q)$, hence
  \begin{align*}
    b(p, \hat{s}_h q) - \hat{b}_h(\hat{p}_h, \hat{s}_h q) = (f - \hat{f}_h J^{-1}_h, \hat{s}_h q).
  \end{align*}
  {\color{black}Again, we can bound this term as follows using \cref{lem:rhs_consistency} and \cref{lem:approximation-DL},
  \begin{align*}
    \vert(f - \hat{f}_hJ^{-1}_h,\hat{s}_hq)\vert \leq ch^{\ell+1}(\Vert f\Vert_0 +\lVert \nabla F\rVert_{L^{\infty}(\mathcal{N})})\Vert q\Vert_{m+1}.
  \end{align*}}
  For the second difference, we use \cref{lem:consistency-laplace-DL} for some $0 < \epsilon < 1/2$,
  \begin{align*}
    \vert \hat{b}_h(\hat{p}_h, \hat{s}_h q) - b(\hat{p}_h, \hat{s}_h q) \vert \leq ch^{\ell} \Vert \hat{p}_h \Vert_{\epsilon} \Vert \hat{s}_h q \Vert_{0}.
  \end{align*}
  Since $\Vert \hat{s}_h q \Vert_{0} \leq \Vert q \Vert_{0} \leq \Vert q \Vert_{m+1}$ and, via \cref{thm:stronger-norm-laplace-DL},
  {\color{black}
  \begin{align*}
    \Vert \hat{p}_h \Vert_{\epsilon} \leq \Vert p \Vert_{\epsilon} + \Vert p - \hat{p}_h \Vert_{\epsilon} \leq c\left[\Vert p\Vert_{\epsilon} + h^{m+1-\epsilon}\Vert f\Vert_{m+1} + h^{\ell-\epsilon}\Vert f\Vert_{1} + h^{\ell+1-\epsilon} (\Vert f\Vert_{0}+\lVert \nabla F \rVert_{L^{\infty}(\mathcal{N})})\right],
  \end{align*}
  we bound $\vert\hat{b}_h(\hat{p}_h,\hat{s}_hq)-b(\hat{p}_h,\hat{s}_hq)\vert$ by
  \begin{align*}
    c\left[h^{\ell}\Vert f\Vert_{1} + h^{m+\ell+1-\epsilon}\Vert f\Vert_{m+1} +  h^{2\ell-\epsilon}\Vert f\Vert_{1} + h^{2\ell+1-\epsilon}(\Vert f\Vert_0 + \lVert \nabla F \rVert_{L^{\infty}(\mathcal{N})})\right]\Vert q\Vert_{m+1}.
  \end{align*}
  Collecting the terms and keeping only the leading-order ones gives the result.}
\end{proof}

\begin{theorem}[Pointwise evaluation]\label{thm:ptwise-eval-laplace-DL}
  Let $p$ and $\hat{p}_h$ denote the solutions to \cref{pb:weak-laplace-DL} and \cref{pb:perturbed-laplace-DL} for sufficiently small $h$. Then for all $\bs{x}\in\Omega$
  \begin{align*}
    \vert u(\bs{x}) - u_h(\bs{x})\vert \leq c_{\bs{x}} \left[h^{2m+2}\Vert f\Vert_{m+1} + {\color{black}h^{\ell}\Vert f\Vert_{1} + h^{\ell+1}( \Vert f \Vert_0 + \lVert \nabla F \rVert_{L^{\infty}(\mathcal{N})})} \right],
  \end{align*}
  with $c_{\bs{x}}\to\infty$ as $\bs{x}\to\Gamma$. 
\end{theorem}

\begin{proof}
  Let $\bs{x}\in\Omega$. We write
  \begin{align*}
    u(\bs{x}) - u_h(\bs{x}) = -\frac{1}{4\pi}\left[\int_\Gamma\frac{(\bs{x}-\bs{y})\cdot\bs{n}(\bs{y})}{\vert\bs{x}-\bs{y}\vert^3}p(\bs{y})-\frac{(\bs{x}-\Psi^{-1}_h(\bs{y}))\cdot\hat{\bs{n}}_h(\bs{y})}{\vert\bs{x}-\Psi^{-1}_h(\bs{y})\vert^3}\hat{p}_h(\bs{y})J^{-1}_h(\bs{y}) \right]d\Gamma(\bs{y}).
  \end{align*}
  We split the difference as follows,
  \begin{align*}
      & u(\bs{x}) - u_h(\bs{x})                                                                                                                                                                                                         \\
    = & -\frac{1}{4\pi}\int_\Gamma\frac{(\bs{x}-\bs{y})\cdot\bs{n}(\bs{y})}{\vert\bs{x}-\bs{y}\vert^3}\left[p(\bs{y}) - \hat{p}_h(\bs{y})\right]d\Gamma(\bs{y})                                                                         \\
      & -\frac{1}{4\pi}\int_\Gamma\left[\frac{(\bs{x}-\bs{y})\cdot\bs{n}(\bs{y})}{\vert\bs{x}-\bs{y}\vert^3} - \frac{(\bs{x}-\bs{y})\cdot\bs{n}(\bs{y})}{\vert\bs{x}-\Psi^{-1}_h(\bs{y})\vert^3}\right]\hat{p}_h(\bs{y})d\Gamma(\bs{y}) \\
      & -\frac{1}{4\pi}\int_\Gamma\frac{(\bs{x}-\bs{y}) - (\bs{x}-\Psi^{-1}_h(\bs{y}))}{\vert\bs{x}-\Psi^{-1}_h(\bs{y})\vert^3}\cdot\bs{n}(\bs{y})\hat{p}_h(\bs{y})d\Gamma(\bs{y})                                                      \\
      & -\frac{1}{4\pi}\int_\Gamma\frac{\bs{x}-\Psi^{-1}_h(\bs{y})}{\vert\bs{x}-\Psi^{-1}_h(\bs{y})\vert^3}\cdot\left[\bs{n}(\bs{y}) - \hat{\bs{n}}_h(\bs{y})\right]\hat{p}_h(\bs{y})d\Gamma(\bs{y})                                    \\
      & -\frac{1}{4\pi}\int_\Gamma\frac{(\bs{x}-\Psi^{-1}_h(\bs{y}))\cdot\hat{\bs{n}}_h(\bs{y})}{\vert\bs{x}-\Psi^{-1}_h(\bs{y})\vert^3}\left[1 - J^{-1}_h(\bs{y})\right]\hat{p}_h(\bs{y})d\Gamma(\bs{y}).
  \end{align*}
  We bound the first term by $c_{\bs{x}}\Vert p-\hat{p}_h\Vert_{-m-1}$, and the others by $c_{\bs{x}}h^{\ell}\Vert\hat{p}_h\Vert_0$ with
  \begin{align*}
    \Vert\hat{p}_h\Vert_{0} \leq \Vert p\Vert_{0}  + \Vert p - \hat{p}_h\Vert_{0} \leq c\left[\Vert p\Vert_{0} + h^{m+1}\Vert f\Vert_{m+1} + h^{\ell}\Vert f\Vert_{1} +{\color{black} h^{\ell+1}(\Vert f\Vert_0+ \lVert \nabla F \rVert_{L^{\infty}(\mathcal{N})})}\right],
  \end{align*}
  and $c_{\bs{x}}\to\infty$ as $\bs{x}\to\Gamma$.
\end{proof}

We conclude this section by stepping back to consider \cref{tab:results}. The convergence rate in \cite{sauter2011} stems from \cite[Cor.~8.2.9]{sauter2011}, where integrals with a \(1/r^2\) singularity are estimated directly, yielding a \(\vert\log h\vert\) term. In our case, such singularities are always multiplied by terms like \(\hat{p}_h(\bs{x}) - \hat{p}_h(\bs{y})\), which regularize the integrand. After all, the double-layer potential on a smooth surface is only weakly singular, and this should be reflected in the analysis.

\section{Convergence rates for the Helmholtz equation}\label{sec:helmholtz}

{\color{black}We emphasize that all functions are now complex-valued. Accordingly, the Sobolev spaces, the \(L^2(\Gamma)\)- and \(L^2(\Gamma_h)\)-inner products, and the discrete spaces \(V_h\) and \(\hat{V}_h\) are all taken over \(\mathbb{C}\), and the forms are sesquilinear.}

\subsection{Single-layer potential}

We consider the following weak formulation of \cref{pb:helmholtz-SL}. We assume that $k^2$ is not a Dirichlet eigenvalue of $-\Delta$ in $\Omega$.

\begin{problem}[Weak formulation]\label{pb:weak-helmholtz-SL}
Find $p\in H^{-1/2}(\Gamma)$ such that
\begin{align*}
  b(p,q) = \langle f, q\rangle \quad \forall q\in H^{-1/2}(\Gamma),
\end{align*}
with $b:H^{-1/2}(\Gamma)\times H^{-1/2}(\Gamma)\to\C$ defined by
\begin{align*}
  b(p,q) = \langle Sp,q\rangle = \frac{1}{4\pi}\int_{\Gamma}\int_{\Gamma}\frac{e^{ik \lvert \bs{x}-\bs{y}\rvert}}{\vert\bs{x}-\bs{y}\vert}p(\bs{y})\overline{q(\bs{x})}d\Gamma(\bs{y})d\Gamma(\bs{x}).
\end{align*}
We assume \( f \in H^{m+2}(\Gamma) \), so that \( p \in H^{m+1}(\Gamma) \), where \( m \geq 0 \) is the polynomial degree in \( V_h \).
\end{problem}

In \cref{pb:weak-helmholtz-SL}, the expression $\langle f, q \rangle$ denotes the duality pairing between $f \in H^{1/2}(\Gamma)$ and $q \in H^{-1/2}(\Gamma)$. When $q\in L^2(\Gamma)$, this pairing is given by the complex $L^2(\Gamma)$-inner product
\[
  \langle f, q \rangle = (f, q) = \int_\Gamma f(\bs{x})\overline{{\color{black}q}(\bs{x})}d\Gamma(\bs{x}) \quad \forall f \in H^{1/2}(\Gamma), \ \forall q \in L^2(\Gamma).
\]
By density, this definition extends uniquely and continuously to all $q \in H^{-1/2}(\Gamma)$.

Since $b$ is injective and can be rewritten as \(b(p,q)=a(p,q)+t(p,q)\) with
\begin{align*}
  a(p,q) & = \frac{1}{4 \pi}\int_{\Gamma} \int_{\Gamma}\frac{p(\bs{y})\overline{q(\bs{x})}}{\lvert\bs{x}-\bs{y}\rvert}d\Gamma(\bs{y})d\Gamma(\bs{x}) \quad \text{(coercive)},                             \\
  t(p,q) & = \frac{1}{4 \pi}\int_{\Gamma} \int_{\Gamma}\frac{e^{ik\lvert\bs{x}-\bs{y}\rvert}-1}{\lvert \bs{x}-\bs{y}\rvert}p(\bs{y})\overline{q(\bs{x})}d\Gamma(\bs{y})d\Gamma(\bs{x}) \quad \text{(compact)},
\end{align*}
there is a unique solution to \cref{pb:weak-helmholtz-SL} by Fredholm's alternative (\cref{thm:fredholm}). We note that the operator associated with \(t\) maps \(H^s(\Gamma)\) to \(H^{s+3}(\Gamma)\){\color{black}, since the expansion
\begin{align*}
  \frac{e^{ik\vert\bs{x}-\bs{y}\vert}-1}{\vert\bs{x}-\bs{y}\vert} = ik - \frac{k^2}{2}\vert\bs{x}-\bs{y}\vert - \frac{ik^3}{6}\vert\bs{x}-\bs{y}\vert^2 + \dots
\end{align*}
shows that its first non-smooth term is a multiple of \(\vert\bs{x}-\bs{y}\vert\), which yields a pseudodifferential operator of order \(-3\) on \(\Gamma\) \cite[Chap.~7]{hsiao2008}}, and is thus compact from \(H^s(\Gamma)\) to \(H^{s+1}(\Gamma)\) for any $s\in\R$. We use the space \(\hat{V}_h\) of \cref{pb:perturbed-laplace-SL} (Laplace single-layer), \textcolor{black}{with complex coefficients.}

\begin{problem}[Perturbed Galerkin approximation problem]\label{pb:perturbed-helmholtz-SL}
Find $\hat{p}_h\in\hat{V}_h$ such that
\begin{align*}
  \hat{b}_h(\hat{p}_h,\hat{q}_h) = (\hat{f}_hJ^{-1}_h,\hat{q}_h) \quad \forall \hat{q}_h\in \hat{V}_h,
\end{align*}
with $\hat{b}_h:H^{-1/2}(\Gamma)\times H^{-1/2}(\Gamma)\to\C$ defined by
\begin{align*}
  \hat{b}_h(\hat{p},\hat{q}) = \frac{1}{4\pi}\int_{\Gamma}\int_{\Gamma}\frac{e^{ik\lvert\Psi^{-1}_h(\bs{x})-\Psi^{-1}_h(\bs{y})\rvert}J^{-1}_h(\bs{y})J^{-1}_h(\bs{x})}{\vert\Psi^{-1}_h(\bs{x})-\Psi^{-1}_h(\bs{y})\vert}\hat{p}(\bs{y})\overline{\hat{q}(\bs{x})}d\Gamma(\bs{y})d\Gamma(\bs{x}).
\end{align*}
\end{problem}

\begin{lemma}[Consistency]\label{lem:consistency-helmholtz-SL}
  There exists \(h_0>0\) such that for all \(h\leq h_0\), the sesquilinear forms defined in \cref{pb:weak-helmholtz-SL} and \cref{pb:perturbed-helmholtz-SL} satisfy the consistency conditions
  \begin{align*}
    \vert b(\hat{p}_h,\hat{q}_h) - \hat{b}_h(\hat{p}_h,\hat{q}_h)\vert \leq c h^{\ell+1} \Vert\hat{p}_h\Vert_{0}\Vert\hat{q}_h\Vert_{0} \quad \forall\hat{p}_h,\hat{q}_h\in\hat{V}_h.
  \end{align*}
\end{lemma}

\begin{proof}
  We write
  \begin{align*}
      & \; b(\hat{p}_h,\hat{q}_h) - \hat{b}_h(\hat{p}_h,\hat{q}_h)                                                                                                                                                                                                                                                           \\
    = & \; \frac{1}{4\pi}\int_{\Gamma}\int_{\Gamma}\hat{p}_h(\bs{y})\overline{\hat{q}_h(\bs{x})}\left[\frac{e^{ik\vert\bs{x}-\bs{y}\vert}}{\vert\bs{x}-\bs{y}\vert} - \frac{e^{ik\vert\Psi^{-1}_h(\bs{x})-\Psi^{-1}_h(\bs{y})\vert}}{\vert\Psi^{-1}_h(\bs{x})-\Psi^{-1}_h(\bs{y})\vert}\right]d\Gamma(\bs{y})d\Gamma(\bs{x}) \\
      & + \; \frac{1}{4\pi}\int_{\Gamma}\int_{\Gamma}\frac{\hat{p}_h(\bs{y})\overline{\hat{q}_h(\bs{x})}e^{ik\vert\Psi^{-1}_h(\bs{x})-\Psi^{-1}_h(\bs{y})\vert}}{\vert\Psi^{-1}_h(\bs{x})-\Psi^{-1}_h(\bs{y})\vert}\left[1 - J^{-1}_h(\bs{y})J^{-1}_h(\bs{x})\right]d\Gamma(\bs{y})d\Gamma(\bs{x}).
  \end{align*}
  This is again straightforward using the geometric estimates in \cref{lem:geometry},
  \begin{align*}
    \vert b(\hat{p}_h,\hat{q}_h) - \hat{b}_h(\hat{p}_h,\hat{q}_h)\vert \leq ch^{\ell+1}\frac{1}{4\pi}\int_{\Gamma}\int_{\Gamma}\frac{\vert\hat{p}_h(\bs{y})\vert\vert\hat{q}_h(\bs{x})\vert}{\vert\bs{x}-\bs{y}\vert}d\Gamma(\bs{y})d\Gamma(\bs{x}),
  \end{align*}
  yielding the result via the continuity of the bilinear form associated with \(S_0\) in $L^2(\Gamma)\times L^2(\Gamma)$.
\end{proof}

Again, using inverse inequalities, we deduce that the $\hat{b}_h$ satisfies the discrete inf-sup conditions uniformly in $H^{-1/2}(\Gamma)\times H^{-1/2}(\Gamma)$ for sufficiently small $h$ via \cref{rem:uniform-infsup}. We can apply Strang's lemma to obtain well-posedness and the following estimate.

\begin{theorem}[Intrinsic norm]\label{thm:intrinsic-norm-helmholtz-SL}
  Let $p$ and $\hat{p}_h$ denote the solutions to \cref{pb:weak-helmholtz-SL} and \cref{pb:perturbed-helmholtz-SL} for sufficiently small $h$. Then
  \begin{align*}
    \Vert p - \hat{p}_h\Vert_{-1/2} \leq c\left[h^{m+3/2}\Vert f\Vert_{m+2} + h^{\ell+1/2}(\Vert f\Vert_{1} {\color{black}+ \lVert \nabla F \rVert_{L^{\infty}(\mathcal{N})}})\right].
  \end{align*}
\end{theorem}

\begin{proof}
  As in \cref{thm:intrinsic-norm-laplace-SL}, the proof relies on {\color{black}\cref{lem:rhs_consistency}}, \cref{lem:consistency-helmholtz-SL}, and \cref{lem:approximation-SL}.
\end{proof}

\begin{theorem}[Stronger norm]\label{thm:stronger-norm-helmholtz-SL}
  Let $p$ and $\hat{p}_h$ denote the solutions to \cref{pb:weak-helmholtz-SL} and \cref{pb:perturbed-helmholtz-SL} for sufficiently small $h$. Then
  \begin{align*}
    \Vert p - \hat{p}_h\Vert_{0} \leq c\left[h^{m+1}\Vert f\Vert_{m+2} + h^{\ell}(\Vert f\Vert_{1} {\color{black}+ \lVert \nabla F \rVert_{L^{\infty}(\mathcal{N})}})\right].
  \end{align*}
\end{theorem}

\begin{proof}
  The result follows from \cref{thm:intrinsic-norm-helmholtz-SL} and \cref{lem:approximation-SL}, as in the proof of \cref{thm:stronger-norm-laplace-SL}.
\end{proof}

\begin{theorem}[Weaker norms]\label{thm:weaker-norm-helmholtz-SL}
  Let $p$ and $\hat{p}_h$ denote the solutions to {\color{black}\cref{pb:weak-helmholtz-SL}} and {\color{black}\cref{pb:perturbed-helmholtz-SL}} for sufficiently small $h$. Then
  \begin{align*}
    \Vert p - \hat{p}_h\Vert_{-m-2} \leq c\left[h^{2m+3}\Vert f\Vert_{m+2} + h^{\ell+1}(\Vert f\Vert_{1} {\color{black}+ \lVert \nabla F \rVert_{L^{\infty}(\mathcal{N})}})\right].
  \end{align*}
\end{theorem}

\begin{proof}
  The proof is similar to that of \cref{thm:weaker-norm-laplace-SL}, and is based on \cref{thm:intrinsic-norm-helmholtz-SL}, \cref{thm:stronger-norm-helmholtz-SL}, and \cref{lem:approximation-SL}. The only difference is that the dual problem reads
  \begin{align*}
    b(p,q) = \langle g, p\rangle \quad \forall p\in H^{-1/2}(\Gamma),
  \end{align*}
  with $b(p,q) = \langle Sp, q\rangle = \overline{\langle S^*q,p\rangle} := b^*(q,p)$. The dual $S^*:H^{-1/2}(\Gamma)\to H^{1/2}(\Gamma)$ is defined by
  \begin{align*}
    (S^*q)(\bs{x}) = \frac{1}{4\pi}\int_\Gamma \frac{e^{-ik\vert\bs{x}-\bs{y}\vert}}{\vert\bs{x}-\bs{y}\vert} q(\bs{y})d\Gamma(\bs{y}),
  \end{align*}
  and shares the same properties as \(S\).
\end{proof}

\begin{theorem}[Pointwise evaluation]\label{thm:ptwise-eval-helmholtz-SL}
  Let $p$ and $\hat{p}_h$ denote the solutions to \cref{pb:weak-helmholtz-SL} and \cref{pb:perturbed-helmholtz-SL} for sufficiently small $h$. Then for all $\bs{x}\in\R^3\setminus\overline{\Omega}$
  \begin{align*}
     & \vert u(\bs{x}) - u_h(\bs{x})\vert \leq c_{\bs{x}} \left[h^{2m+3}\Vert f\Vert_{m+2} + h^{\ell+1}(\Vert f\Vert_{1} {\color{black}+ \lVert \nabla F \rVert_{L^{\infty}(\mathcal{N})}})\right],
  \end{align*}
  with $c_{\bs{x}}\to\infty$ as $\bs{x}\to\Gamma$.
\end{theorem}

\begin{proof}
  Let $\bs{x}\in\R^3\setminus\overline{\Omega}$. We write
  \begin{align*}
    u(\bs{x}) - u_h(\bs{x}) = \frac{1}{4\pi}\int_\Gamma\left[\frac{e^{ik\vert\bs{x}-\bs{y}\vert}}{\vert\bs{x}-\bs{y}\vert}p(\bs{y})-\frac{e^{ik\vert\bs{x}-\Psi^{-1}_h(\bs{y})\vert}}{\vert\bs{x}-\Psi^{-1}_h(\bs{y})\vert}\hat{p}_h(\bs{y})J^{-1}_h(\bs{y})\right]d\Gamma(\bs{y}),
  \end{align*}
  which we split into
  \begin{align*}
    u(\bs{x}) - u_h(\bs{x}) = & \; \frac{1}{4\pi}\int_\Gamma\left[\frac{e^{ik\vert\bs{x}-\bs{y}\vert}}{\vert\bs{x}-\bs{y}\vert}p(\bs{y})-\frac{e^{ik\vert\bs{x}-\bs{y}\vert}}{\vert\bs{x}-\bs{y}\vert}\hat{p}_h(\bs{y})\right]d\Gamma(\bs{y})                      \\
                              & \; + \frac{1}{4\pi}\int_\Gamma\hat{p}_h(\bs{y})\left[\frac{e^{ik\vert\bs{x}-\bs{y}\vert}}{\vert\bs{x}-\bs{y}\vert}-\frac{e^{ik\vert\bs{x} -\Psi^{-1}_h(\bs{y})\vert}}{\vert\bs{x} -\Psi^{-1}_h(\bs{y})\vert}\right]d\Gamma(\bs{y}) \\
                              & \; + \frac{1}{4\pi}\int_\Gamma\hat{p}_h(\bs{y})\frac{e^{ik\vert\bs{x}-\Psi^{-1}_h(\bs{y})\vert}}{\vert\bs{x} -\Psi^{-1}_h(\bs{y})\vert}(1 - J^{-1}_h(\bs{y}))d\Gamma(\bs{y}).
  \end{align*}
  We bound the first term by $c_{\bs{x}}\Vert p-\hat{p}_h\Vert_{-m-2}$ and the rest by $c_{\bs{x}}h^{\ell+1}\Vert\hat{p}_h\Vert_{0}$, with $c_{\bs{x}}\to\infty$ as $\bs{x}\to\Gamma$.
\end{proof}

\subsection{Double-layer potential}

We consider the following weak formulation of \cref{pb:helmholtz-DL}. We assume that $k^2$ is not a Neumann eigenvalue of $-\Delta$ in $\Omega$.

\begin{problem}[Weak formulation]\label{pb:weak-helmholtz-DL}
Find $p\in L^{2}(\Gamma)$ such that
\begin{align*}
  b(p,q) = (f, q) \quad \forall q\in L^{2}(\Gamma),
\end{align*}
with $b:L^{2}(\Gamma)\times L^{2}(\Gamma)\to\C$ defined by $b(p,q) = \frac{1}{2}(p,q) + (Dp,q)$, that is,
\begin{align*}
  b(p,q) = \frac{1}{2} \int_{\Gamma} p(\bs{x})\overline{q(\bs{x})} d\Gamma(\bs{x}) +\frac{1}{4\pi}\int_\Gamma \int_{\Gamma}\frac{\partial }{\partial \bs{n}(\bs{y})} \left(\frac{e^{ik\vert\bs{x}-\bs{y}\vert}}{\vert\bs{x}-\bs{y}\vert}\right) p(\bs{y})\overline{q(\bs{x})} d\Gamma(\bs{y})d\Gamma(\bs{x}).
\end{align*}
We assume \( f \in H^{m+1}(\Gamma) \), so that \( p \in H^{m+1}(\Gamma) \), where \( m \geq 0 \) is the polynomial degree in \( V_h \).
\end{problem}

Let $r$ be the difference between the double-layer operator for Helmholtz and Laplace, i.e.,
\begin{align*}
  r(p,q)= ((D-D_0)p,q) = \frac{1}{4\pi}\int_\Gamma \int_{\Gamma}\frac{\partial }{\partial \bs{n}(\bs{y})}\left(\frac{e^{ik\vert\bs{x}-\bs{y}\vert}-1}{\vert\bs{x}-\bs{y}\vert}\right) p(\bs{y})\overline{q(\bs{x})} d\Gamma(\bs{y})d\Gamma(\bs{x}).
\end{align*}
Since $b$ is injective and can be rewritten as $b(p,q)=a(p,q)+t(p,q)$ with
\begin{align*}
  a(p,q) =\frac{1}{2}(p,q) \quad \text{(coercive)}, \qquad t(p,q) = r(p,q) + (D_0p,q) \quad \text{(compact)},
\end{align*}
there is a unique solution to \cref{pb:weak-helmholtz-DL} by Fredholm's alternative (\cref{thm:fredholm}). Note that both $D-D_0$ and $D_0$ are compact in $L^2(\Gamma)$ since $\Gamma$ is smooth \cite{chandler2012}. Here, the approximation space $\hat{V}_h$ is the same as for \cref{pb:perturbed-laplace-DL} (the Laplace double-layer case), \textcolor{black}{with complex coefficients}.
{\color{black}
Using the identity \cite[eq.~(3.34)]{nedelec1977} as in the Laplace case, we rewrite \(b\) as
\begin{align*}
  b(p,q) = -\frac{1}{4\pi}\int_\Gamma\int_\Gamma (p(\bs{x}) - p(\bs{y})) \overline{q(\bs{x})}\frac{(\bs{x}-\bs{y})\cdot \bs{n}(\bs{y})}{\vert\bs{x}-\bs{y}\vert^3} d\Gamma(\bs{y})d\Gamma(\bs{x}) + r(p,q).
\end{align*}
The first term is the second term of the Laplace form~\cref{eq:laplace_bilinear_form_rewritten} with the opposite sign; the \(L^2(\Gamma)\)-inner product term is absent.}

\begin{problem}[Perturbed Galerkin approximation problem]\label{pb:perturbed-helmholtz-DL}
Find $\hat{p}_h\in\hat{V}_h$ such that
\begin{align*}
  \hat{b}_h(\hat{p}_h,\hat{q}_h) = (\hat{f}_h J^{-1}_h, \hat{q}_h) \quad \forall \hat{q}_h\in \hat{V}_h,
\end{align*}{\color{black}
with $\hat{b}_h:L^{2}(\Gamma)\times L^{2}(\Gamma)\to\C$ defined by
\begin{align*}
  \hat{b}_h(\hat{p},\hat{q}) = & -\frac{1}{4\pi}\int_{\Gamma} \int_{\Gamma} (\hat{p}(\bs{x}) - \hat{p}(\bs{y})) \overline{\hat{q}(\bs{x})}\frac{(\Psi^{-1}_h(\bs{x})-\Psi^{-1}_h(\bs{y}))\cdot \hat{\bs{n}}_h(\bs{y})}{\vert\Psi^{-1}_h(\bs{x})-\Psi^{-1}_h(\bs{y})\vert^3}J^{-1}_h(\bs{y})J^{-1}_h(\bs{x}) d\Gamma(\bs{y})d\Gamma(\bs{x}) \\
                               & + \hat{r}_h(\hat{p},\hat{q}).
\end{align*}
The first term is the second term of the Laplace form in~\cref{pb:perturbed-laplace-DL} with the opposite sign, and}
\begin{align*}
  \hat{r}_h(\hat{p},\hat{q}) = & \; \frac{1}{4\pi}\int_{\Gamma}\int_{\Gamma}\frac{(1-ik\vert \Psi^{-1}_h(\bs{x})-\Psi^{-1}_h(\bs{y})\vert)e^{ik\vert\Psi^{-1}_h(\bs{x})-\Psi^{-1}_h(\bs{y})\vert}-1}{\vert\Psi^{-1}_h(\bs{x})-\Psi^{-1}_h(\bs{y})\vert^3} \\
                               & \; \times (\Psi^{-1}_h(\bs{x})-\Psi^{-1}_h(\bs{y})) \cdot \hat{\bs{n}}_h(\bs{y}) \hat{p}(\bs{y}) \overline{\hat{q}(\bs{x})} J^{-1}_h(\bs{y})J^{-1}_h(\bs{x}) d\Gamma(\bs{y})d\Gamma(\bs{x}).
\end{align*}
\end{problem}

\begin{lemma}[Consistency]\label{lem:consistency-helmholtz-DL}
  There exists \(h_0>0\) such that for all \(h\leq h_0\), the sesquilinear forms defined in \cref{pb:weak-helmholtz-DL} and \cref{pb:perturbed-helmholtz-DL} satisfy the consistency conditions
    \begin{align*}
    \vert b(\hat{p}_h,\hat{q}_h) - \hat{b}_h(\hat{p}_h,\hat{q}_h)\vert \leq c_\epsilon h^{\ell} \Vert\hat{p}_h\Vert_{\epsilon}\Vert\hat{q}_h\Vert_{0} \quad {\color{black}\forall \epsilon\in(0,1/2),} \; \forall\hat{p}_h,\hat{q}_h\in\hat{V}_h,
  \end{align*}
  with $c_\epsilon\to\infty$ as $\epsilon\to0$. 
\end{lemma}

\begin{proof}{\color{black}
  Adding and subtracting $(\hat{p}_h,\hat{q}_h)$ and $(\hat{p}_h,\hat{q}_hJ^{-1}_h)$, we can write $b(\hat{p}_h,\hat{q}_h) - \hat{b}_h(\hat{p}_h,\hat{q}_h)$ as the sum of three terms: $(\hat{p}_h,\hat{q}_h(1-J^{-1}_h))$, minus the difference between the forms of \cref{pb:weak-laplace-DL} and \cref{pb:perturbed-laplace-DL}, and $r(\hat{p}_h,\hat{q}_h) - \hat{r}_h(\hat{p}_h,\hat{q}_h)$. The first term is bounded using \cref{lem:geometry}, and the second using the consistency result for the Laplace problem (\cref{lem:consistency-laplace-DL}). This gives
  \begin{align*}
    \vert b(\hat{p}_h,\hat{q}_h) - \hat{b}_h(\hat{p}_h,\hat{q}_h)\vert \leq c_\epsilon h^{\ell} \Vert\hat{p}_h\Vert_{\epsilon}\Vert\hat{q}_h\Vert_{0} +\vert r(\hat{p}_h,\hat{q}_h) - \hat{r}_h(\hat{p}_h,\hat{q}_h)\vert,
  \end{align*}}
  with $r(\hat{p}_h,\hat{q}_h) - \hat{r}_h(\hat{p}_h,\hat{q}_h) = \frac{1}{4\pi}\int_{\Gamma}\int_{\Gamma}\delta_h(\bs{x},\bs{y})\hat{p}_h(\bs{y})\overline{\hat{q}_h(\bs{x})}d\Gamma(\bs{y})d\Gamma(\bs{x})$, and
  \begin{align*}
     & \delta_h(\bs{x},\bs{y}) = s(\bs{x},\bs{y})(\bs{x}-\bs{y})\cdot\bs{n}(\bs{y}) - s_h(\bs{x},\bs{y})(\Psi^{-1}_h(\bs{x})-\Psi^{-1}_h(\bs{y}))\cdot\hat{\bs{n}}_h(\bs{y})J^{-1}_h(\bs{y})J^{-1}_h(\bs{x}), \\
     & s(\bs{x},\bs{y}) = \frac{(1-ik\vert\bs{x}-\bs{y}\vert)e^{ik\vert\bs{x}-\bs{y}\vert}-1}{\vert\bs{x}-\bs{y}\vert^3}, \\
     & s_h(\bs{x},\bs{y}) = \frac{(1-ik\vert\Psi^{-1}_h(\bs{x})-\Psi^{-1}_h(\bs{y})\vert)e^{ik\vert\Psi^{-1}_h(\bs{x})-\Psi^{-1}_h(\bs{y})\vert}-1}{\vert\Psi^{-1}_h(\bs{x})-\Psi^{-1}_h(\bs{y})\vert^3}.
  \end{align*}
  We write
  \begin{align*}
    \delta_h(\bs{x},\bs{y}) = & \; s(\bs{x},\bs{y})\left[(\bs{x}-\bs{y})\cdot\bs{n}(\bs{y}) - (\Psi^{-1}_h(\bs{x})-\Psi^{-1}_h(\bs{y}))\cdot\hat{\bs{n}}_h(\bs{y})\right]               \\
                            & \; + s(\bs{x},\bs{y})(\Psi^{-1}_h(\bs{x})-\Psi^{-1}_h(\bs{y}))\cdot\hat{\bs{n}}_h(\bs{y}) [1 - J^{-1}_h(\bs{y})J^{-1}_h(\bs{x})]                        \\
                            & \; + (\Psi^{-1}_h(\bs{x})-\Psi^{-1}_h(\bs{y}))\cdot\hat{\bs{n}}_h(\bs{y}) J^{-1}_h(\bs{y})J^{-1}_h(\bs{x}) [s(\bs{x},\bs{y})-s_h(\bs{x},\bs{y})].
  \end{align*}
  The first two terms can be bounded by {\color{black}\(c h^{\ell}\lvert\bs{x}-\bs{y}\rvert^{-1}\)} using \cref{lem:geometry} and $\lvert s(\bs{x},\bs{y})\rvert \leq c\lvert\bs{x}-\bs{y}\rvert^{-2}$. Using the continuity of the Laplace single-layer in \(L^2(\Gamma)\times L^2(\Gamma)\), we obtain consistency of the first two terms in \(L^2(\Gamma)\times L^2(\Gamma)\). {\color{black}For the third term, we show below that}
  \begin{align*}
    \lvert s(\bs{x},\bs{y})-s_h(\bs{x},\bs{y})\rvert\leq c \frac{h^{\ell+1}}{\lvert \bs{x} - \bs{y}\rvert^2}{\color{black}.}
  \end{align*}
  {\color{black}The kernel $\lvert\bs{x}-\bs{y}\rvert^{-2}$ is not integrable on $\Gamma$, but it is multiplied by $(\Psi^{-1}_h(\bs{x})-\Psi^{-1}_h(\bs{y}))\cdot\hat{\bs{n}}_h(\bs{y})$, which is bounded by $c\lvert\bs{x}-\bs{y}\rvert$ by \cref{lem:geometry}. The third term is therefore bounded by $ch^{\ell+1}\lvert\bs{x}-\bs{y}\rvert^{-1}$, i.e., by $h^{\ell+1}$ times the kernel of the Laplace single-layer operator, and we conclude as for the first two terms. To prove the bound on $s-s_h$, we write}
  \begin{align*}
    s(\bs{x},\bs{y})-s_h(\bs{x},\bs{y}) = & \; (e^{ik\vert\bs{x}-\bs{y}\vert}-1)\left(\frac{1}{\lvert\bs{x}-\bs{y} \rvert^3}-\frac{1}{\lvert\Psi^{-1}_h(\bs{x})-\Psi^{-1}_h(\bs{y}) \rvert^3}\right)                  \\
                                                & \; + \frac{1}{\lvert\Psi^{-1}_h(\bs{x})-\Psi^{-1}_h(\bs{y}) \rvert^3} \left(e^{ik\vert\bs{x}-\bs{y}\vert}-e^{ik\vert\Psi^{-1}_h(\bs{x})-\Psi^{-1}_h(\bs{y})\vert}\right)  \\
                                                & \; - ik e^{ik\vert\bs{x}-\bs{y}\vert} \left( \frac{1}{\lvert\bs{x}-\bs{y} \rvert^2}-\frac{1}{\lvert \Psi^{-1}_h(\bs{x})-\Psi^{-1}_h(\bs{y}) \rvert^2} \right)             \\
                                                & \; - \frac{ik}{\lvert\Psi^{-1}_h(\bs{x})-\Psi^{-1}_h(\bs{y}) \rvert^2}\left(e^{ik\vert\bs{x}-\bs{y}\vert}-e^{ik\vert\Psi^{-1}_h(\bs{x})-\Psi^{-1}_h(\bs{y})\vert}\right),
  \end{align*}
  where all four terms can be bounded by {\color{black}\(ch^{\ell+1}\lvert\bs{x} - \bs{y}\rvert^{-2}\)} using \cref{lem:geometry}.
\end{proof}

\begin{theorem}[Intrinsic norm]\label{thm:intrinsic-norm-helmholtz-DL}
  Let $p$ and $\hat{p}_h$ denote the solutions to \cref{pb:weak-helmholtz-DL} and \cref{pb:perturbed-helmholtz-DL} for sufficiently small $h$. Then
  \begin{align*}
    \Vert p - \hat{p}_h\Vert_{0} \leq c\big[h^{m+1}\Vert f\Vert_{m+1} + {\color{black}h^{\ell}\Vert f\Vert_{1} + h^{\ell+1}(\Vert f\Vert_0 + \lVert \nabla F \rVert_{L^{\infty}(\mathcal{N})})}\big].
  \end{align*}
\end{theorem}

\begin{proof}
  We combine \cref{lem:rhs_consistency,lem:consistency-helmholtz-DL,lem:approximation-DL} with \cref{lem:strang-B}, as in the proof of \cref{thm:intrinsic-norm-laplace-DL}.
\end{proof}

\begin{theorem}[Stronger norm]\label{thm:stronger-norm-helmholtz-DL}
  Let $p$ and $\hat{p}_h$ denote the solutions to \cref{pb:weak-helmholtz-DL} and \cref{pb:perturbed-helmholtz-DL} for sufficiently small $h$. Then
  \color{black}{  
  \begin{align*}
    \Vert p - \hat{p}_h\Vert_{\epsilon} \leq c\left[h^{m+1-\epsilon}\Vert f\Vert_{m+1} + h^{\ell-\epsilon}\Vert f\Vert_{1} + h^{\ell+1-\epsilon}( \Vert f\Vert_{0}+ \lVert \nabla F \rVert_{L^{\infty}(\mathcal{N})})\right] \qquad \forall\epsilon\in(0,1/2).
  \end{align*}}
\end{theorem}

\begin{proof}
  The result follows from \cref{thm:intrinsic-norm-helmholtz-DL} and \cref{lem:approximation-DL}, as in the proof of \cref{thm:stronger-norm-laplace-DL}.
\end{proof}

\begin{theorem}[Weaker norms]\label{thm:weaker-norm-helmholtz-DL}
  Let $p$ and $\hat{p}_h$ denote the solutions to \cref{pb:weak-helmholtz-DL} and \cref{pb:perturbed-helmholtz-DL} for sufficiently small $h$. Then
  \begin{align*}
    \Vert p - \hat{p}_h\Vert_{-m-1} \leq c\left[h^{2m+2}\Vert f\Vert_{m+1} + {\color{black}h^{\ell}\Vert f\Vert_{1} + h^{\ell+1}(\Vert f \Vert_0 + \lVert \nabla F \rVert_{L^{\infty}(\mathcal{N})})}\right].
  \end{align*}
\end{theorem}

\begin{proof}
  The proof is similar to that of \cref{thm:weaker-norm-laplace-DL}, and is based on \cref{thm:intrinsic-norm-helmholtz-DL}, \cref{thm:stronger-norm-helmholtz-DL}, and \cref{lem:approximation-DL}. The only difference is that the dual problem reads
  \begin{align*}
    b(p,q) = (p,g) \quad \forall p\in L^2(\Gamma),
  \end{align*}
  with $b(p,q) = (p,q)/2 + (Dp,q) = \overline{(q,p)}/2 + \overline{(D^*q, p)} := b^*(q,p)$ with dual $D^*$ given by
  \begin{align*}
    (D^*q)(\bs{x}) = \frac{1}{4\pi}\int_\Gamma \frac{\partial }{\partial \bs{n}(\bs{x})} \left(\frac{e^{-ik\lvert \bs{x}-\bs{y}\rvert}}{\vert\bs{x}-\bs{y}\vert}\right) q(\bs{y}) d\Gamma(\bs{y}).
  \end{align*}
  {\color{black}Since $k^2$ is not a Neumann eigenvalue of $-\Delta$ in $\Omega$, $I/2+D$ is an isomorphism on $H^s(\Gamma)$ for all $s\in\R$ \cite[Thm.~6.34]{steinbach2008}, and so is its adjoint $I/2+D^*$ by duality. Hence $g\in H^{m+1}(\Gamma)$ implies $q\in H^{m+1}(\Gamma)$ with $\Vert q\Vert_{m+1}\leq c\Vert g\Vert_{m+1}$.}
\end{proof}

\begin{theorem}[Pointwise evaluation]\label{thm:ptwise-eval-helmholtz-DL}
  Let $p$ and $\hat{p}_h$ denote the solutions to \cref{pb:weak-helmholtz-DL} and \cref{pb:perturbed-helmholtz-DL} for sufficiently small $h$. Then for all $\bs{x}\in\R^3\setminus\overline{\Omega}$
  \begin{align*}
    \vert u(\bs{x}) - u_h(\bs{x})\vert \leq c_{\bs{x}}\left[h^{2m+2}\Vert f\Vert_{m+1} + {\color{black}h^{\ell}\Vert f\Vert_{1} + h^{\ell+1}(\Vert f\Vert_0 + \lVert \nabla F\rVert_{L^{\infty}(\mathcal{N})})}\right],
  \end{align*}
  with $c_{\bs{x}}\to\infty$ as $\bs{x}\to\Gamma$.
\end{theorem}

\begin{proof}
  Let $\bs{x}\in\R^3\setminus\overline{\Omega}$. We write
  \begin{align*}
    u(\bs{x}) = \frac{1}{4\pi}\int_\Gamma(1-ik \vert\bs{x}-\bs{y}\vert)\frac{e^{ik \vert\bs{x}-\bs{y}\vert}}{\vert\bs{x}-\bs{y}\vert^3}(\bs{x}-\bs{y})\cdot\bs{n}(\bs{y})p(\bs{y})d\Gamma(\bs{y}),
  \end{align*}
  and
  \begin{align*}
    u_h(\bs{x}) = \frac{1}{4\pi}\int_\Gamma(1-ik \vert\bs{x}-\Psi^{-1}_h(\bs{y})\vert)\frac{e^{ik \vert\bs{x}-\Psi^{-1}_h(\bs{y})\vert}}{\vert\bs{x}-\Psi^{-1}_h(\bs{y})\vert^3}\hat{p}_h(\bs{y})(\bs{x}-\Psi^{-1}_h(\bs{y}))\cdot\hat{\bs{n}}_h(\bs{y}) J^{-1}_h(\bs{y}) d\Gamma(\bs{y}).
  \end{align*}
  Let
  \begin{align*}
     & v(\bs{x}) = \frac{1}{4\pi}\int_\Gamma \frac{e^{ik \vert\bs{x}-\bs{y}\vert}}{\vert\bs{x}-\bs{y}\vert^3}(\bs{x}-\bs{y})\cdot\bs{n}(\bs{y})p(\bs{y})d\Gamma(\bs{y}),                                                                                \\
     & v_h(\bs{x}) = \frac{1}{4\pi}\int_\Gamma  \frac{e^{ik \vert\bs{x}-\Psi^{-1}_h(\bs{y})\vert}}{\vert\bs{x}-\Psi^{-1}_h(\bs{y})\vert^3}\hat{p}_h(\bs{y})(\bs{x}-\Psi^{-1}_h(\bs{y}))\cdot\hat{\bs{n}}_h(\bs{y}) J^{-1}_h(\bs{y}) d\Gamma(\bs{y}),    \\
     & w(\bs{x}) = -\frac{ik}{4\pi}\int_\Gamma \frac{e^{ik \vert\bs{x}-\bs{y}\vert}}{\vert\bs{x}-\bs{y}\vert^2}(\bs{x}-\bs{y})\cdot\bs{n}(\bs{y})p(\bs{y})d\Gamma(\bs{y}),                                                                              \\
     & w_h(\bs{x}) = -\frac{ik}{4\pi}\int_\Gamma   \frac{e^{ik \vert\bs{x}-\Psi^{-1}_h(\bs{y})\vert}}{\vert\bs{x}-\Psi^{-1}_h(\bs{y})\vert^2}\hat{p}_h(\bs{y})(\bs{x}-\Psi^{-1}_h(\bs{y}))\cdot\hat{\bs{n}}_h(\bs{y}) J^{-1}_h(\bs{y}) d\Gamma(\bs{y}),
  \end{align*}
  so that \(u(\bs{x})=v(\bs{x})+w(\bs{x})\) and \(u_h(\bs{x})=v_h(\bs{x})+w_h(\bs{x})\). We split the first difference,
  \begin{align*}
      & \; v(\bs{x}) - v_h(\bs{x})                                                                                                                                                                                                                                                                                         \\
    = & \; \frac{1}{4\pi}\int_\Gamma e^{ik \lvert \bs{x}-\bs{y}\rvert}\left( \frac{(\bs{x}-\bs{y})\cdot\bs{n}(\bs{y})p(\bs{y})}{\vert\bs{x}-\bs{y}\vert^3}-\frac{(\bs{x}-\Psi^{-1}_h(\bs{y}))\cdot\hat{\bs{n}}_h(\bs{y})\hat{p}_h(\bs{y})J^{-1}_h(\bs{y})}{\vert\bs{x}-\Psi^{-1}_h(\bs{y})\vert^3} \right) d\Gamma(\bs{y}) \\
      & \; + \frac{1}{4\pi}\int_\Gamma \frac{(\bs{x}-\Psi^{-1}_h(\bs{y}))\cdot\hat{\bs{n}}_h(\bs{y})\hat{p}_h(\bs{y})J^{-1}_h(\bs{y})}{\vert\bs{x}-\Psi^{-1}_h(\bs{y})\vert^3}\left(e^{ik \lvert \bs{x}-\bs{y}\rvert} - e^{ik \lvert \bs{x}-\Psi^{-1}_h(\bs{y})\rvert}\right)d\Gamma(\bs{y}).
  \end{align*}
  We bound the first term by $c_{\bs{x}}\Vert p-\hat{p}_h\Vert_{-m-1}$, noting that the difference is exactly the one for the Laplace double-layer problem (\cref{pb:laplace-DL}), and the second term by $c_{\bs{x}}h^{\ell+1}\Vert\hat{p}_h\Vert_0$ with
  \begin{align*}
    \Vert\hat{p}_h\Vert_{0} \leq \Vert p\Vert_{0}  + \Vert p - \hat{p}_h\Vert_{0} \leq c\left[\Vert p\Vert_{0} + h^{m+1}\Vert f\Vert_{m+1} + {\color{black}h^{\ell}\Vert f\Vert_{1} + h^{\ell+1}(\Vert f\Vert_0 + \lVert \nabla F\rVert_{L^{\infty}(\mathcal{N})})}\right],
  \end{align*}
  and $c_{\bs{x}}\to\infty$ as $\bs{x}\to\Gamma$. Similarly, we split the second difference,
  \begin{align*}
    \hspace{-0.05cm}   & w(\bs{x}) - w_h(\bs{x})                                                                                                                                                                                                                                                                                                                                    \\
    \hspace{-0.05cm} = & - \frac{ik}{4\pi}\int_\Gamma \lvert \bs{x} - \bs{y}\rvert e^{ik \lvert \bs{x}-\bs{y}\rvert}\left( \frac{(\bs{x}-\bs{y})\cdot\bs{n}(\bs{y})p(\bs{y})}{\vert\bs{x}-\bs{y}\vert^3}-\frac{(\bs{x}-\Psi^{-1}_h(\bs{y}))\cdot\hat{\bs{n}}_h(\bs{y})\hat{p}_h(\bs{y})J^{-1}_h(\bs{y})}{\vert\bs{x}-\Psi^{-1}_h(\bs{y})\vert^3} \right) d\Gamma(\bs{y})            \\
    \hspace{-0.05cm}   & - \frac{ik}{4\pi}\int_\Gamma \frac{(\bs{x}-\Psi^{-1}_h(\bs{y}))\cdot\hat{\bs{n}}_h(\bs{y})\hat{p}_h(\bs{y})J^{-1}_h(\bs{y})}{\vert\bs{x}-\Psi^{-1}_h(\bs{y})\vert^3}\left(\lvert \bs{x} - \bs{y}\rvert e^{ik \lvert \bs{x}-\bs{y}\rvert} - \lvert \bs{x} - \Psi^{-1}_h(\bs{y})\rvert e^{ik \lvert \bs{x}-\Psi^{-1}_h(\bs{y})\rvert}\right)d\Gamma(\bs{y}).
  \end{align*}
  Both terms can be bounded as before.
\end{proof}

\subsection{CFIE}

We consider the following weak formulation of \cref{pb:helmholtz-CFIE}.
\begin{problem}[Weak formulation]\label{pb:weak-helmholtz-CFIE}
Find $p\in L^{2}(\Gamma)$ such that
\begin{align*}
  b(p,q) = (f, q) \quad \forall q\in L^{2}(\Gamma),
\end{align*}
with $b:L^{2}(\Gamma)\times L^{2}(\Gamma)\to\C$ defined by the formula $b(p,q) = \frac{1}{2}(p,q) + (Dp,q) - i \eta (Sp,q)$. We assume \( f \in H^{m+1}(\Gamma) \), so that \( p \in H^{m+1}(\Gamma) \), where \( m \geq 0 \) is the polynomial degree in \( V_h \).
\end{problem}

Using \(D_0\) and \(r\) as in \cref{pb:weak-helmholtz-DL}, and noting that \(S:L^{2}(\Gamma)\to L^{2}(\Gamma)\) is compact, \(b\), which is injective, can be rewritten as \(b(p,q)=a(p,q)+t(p,q)\) with
\begin{align*}
  a(p,q) =\frac{1}{2}(p,q) \quad \text{(coercive)}, \qquad t(p,q) = r(p,q) + (D_0p,q)-i\eta (Sp,q) \quad \text{(compact)}.
\end{align*}
Hence, there is a unique solution to \cref{pb:weak-helmholtz-CFIE} via Fredholm's alternative (\cref{thm:fredholm}). The approximation space $\hat{V}_h$ is {\color{black}the} same as for \cref{pb:perturbed-helmholtz-DL}.

\begin{problem}[Perturbed Galerkin approximation problem]\label{pb:perturbed-helmholtz-CFIE}
Find $\hat{p}_h\in\hat{V}_h$ such that
\begin{align*}
  \hat{b}_h(\hat{p}_h,\hat{q}_h) = (\hat{f}_h J^{-1}_h, \hat{q}_h) \quad \forall \hat{q}_h\in \hat{V}_h,
\end{align*}
with $\hat{b}_h:L^{2}(\Gamma)\times L^{2}(\Gamma)\to\C$ defined as a linear combination of the sesquilinear forms from \cref{pb:perturbed-helmholtz-SL} and \cref{pb:perturbed-helmholtz-DL}.
\end{problem}

\begin{lemma}[Consistency]\label{lem:consistency-helmholtz-CFIE}
  There exists \(h_0>0\) such that for all \(h\leq h_0\), the sesquilinear forms defined in \cref{pb:weak-helmholtz-CFIE} and \cref{pb:perturbed-helmholtz-CFIE} satisfy the consistency conditions
    \begin{align*}
    \vert b(\hat{p}_h,\hat{q}_h) - \hat{b}_h(\hat{p}_h,\hat{q}_h)\vert \leq c_\epsilon h^{\ell} \Vert\hat{p}_h\Vert_{\epsilon}\Vert\hat{q}_h\Vert_{0} \quad {\color{black}\forall \epsilon\in(0,1/2),} \; \forall\hat{p}_h,\hat{q}_h\in\hat{V}_h,
  \end{align*}
  with $c_\epsilon\to\infty$ as $\epsilon\to0$.
\end{lemma}
\begin{proof}
  We combine \cref{lem:consistency-helmholtz-SL} with \cref{lem:consistency-helmholtz-DL}.
\end{proof}

\begin{theorem}[Intrinsic norm]\label{thm:intrinsic-norm-helmholtz-CFIE}
  Let $p$ and $\hat{p}_h$ denote the solutions to \cref{pb:weak-helmholtz-CFIE} and \cref{pb:perturbed-helmholtz-CFIE} for sufficiently small $h$. Then
  \begin{align*}
    \Vert p - \hat{p}_h\Vert_{0} \leq c\big[h^{m+1}\Vert f\Vert_{m+1} + {\color{black} h^{\ell}\Vert f\Vert_{1} + h^{\ell+1} ( \Vert f \Vert_0 + \lVert \nabla F\rVert_{L^{\infty}(\mathcal{N})})}\big].
  \end{align*}
\end{theorem}

\begin{proof}
  We combine \cref{lem:rhs_consistency,lem:consistency-helmholtz-CFIE,lem:approximation-DL} with \cref{lem:strang-B}, as in the proof of \cref{thm:intrinsic-norm-helmholtz-DL}.
\end{proof}

\begin{theorem}[Stronger norm]\label{thm:stronger-norm-helmholtz-CFIE}
  Let $p$ and $\hat{p}_h$ denote the solutions to \cref{pb:weak-helmholtz-CFIE} and \cref{pb:perturbed-helmholtz-CFIE} for sufficiently small $h$. Then
  \color{black}{  
  \begin{align*}
    \Vert p - \hat{p}_h\Vert_{\epsilon} \leq c\left[h^{m+1-\epsilon}\Vert f\Vert_{m+1} + h^{\ell-\epsilon}\Vert f\Vert_{1} + h^{\ell+1-\epsilon}( \Vert f\Vert_{0}+ \lVert \nabla F \rVert_{L^{\infty}(\mathcal{N})})\right] \qquad \forall\epsilon\in(0,1/2).
  \end{align*}}
\end{theorem}

\begin{proof}
  The result follows from \cref{thm:intrinsic-norm-helmholtz-CFIE} and \cref{lem:approximation-DL}, as in the proof of \cref{thm:stronger-norm-helmholtz-DL}.
\end{proof}

\begin{theorem}[Weaker norms]\label{thm:weaker-norm-helmholtz-CFIE}
  Let $p$ and $\hat{p}_h$ denote the solutions to \cref{pb:weak-helmholtz-CFIE} and \cref{pb:perturbed-helmholtz-CFIE} for sufficiently small $h$. Then
  \begin{align*}
    \Vert p - \hat{p}_h\Vert_{-m-1} \leq c\left[h^{2m+2}\Vert f\Vert_{m+1} + {\color{black}h^{\ell}\Vert f\Vert_{1} + h^{\ell+1}( \Vert f\Vert_0 + \lVert \nabla F\rVert_{L^{\infty}(\mathcal{N})})}\right].
  \end{align*}
\end{theorem}

\begin{proof}
  The proof is similar to that of \cref{thm:weaker-norm-laplace-DL}, and is based on \cref{thm:intrinsic-norm-helmholtz-CFIE}, \cref{thm:stronger-norm-helmholtz-CFIE}, and \cref{lem:approximation-DL}. The only difference is that the dual problem reads
  \begin{align*}
    b(p,q) = (p,g) \quad \forall p\in L^2(\Gamma),
  \end{align*}
  with
  \begin{align*}
    b(p,q) = (p,q)/2 + {\color{black}(Dp,q) - i \eta (Sp,q)} = \overline{(q,p)}/2 + \overline{(D^*q, p)} - i \eta \overline{(S^*q, p)}  := b^*(q,p),
  \end{align*}
  where the dual operators are defined as in the proofs of \cref{thm:weaker-norm-helmholtz-SL} and \cref{thm:weaker-norm-helmholtz-DL}. {\color{black}Since $I/2+D-i\eta S$ is an isomorphism on $H^s(\Gamma)$ for all $s\in\R$ \cite[Thm.~6.34]{steinbach2008}, so is its adjoint $I/2+D^*+i\eta S^*$ by duality. Hence $g\in H^{m+1}(\Gamma)$ implies $q\in H^{m+1}(\Gamma)$ with $\Vert q\Vert_{m+1}\leq c\Vert g\Vert_{m+1}$.}
\end{proof}

\begin{theorem}[Pointwise evaluation]\label{thm:ptwise-eval-helmholtz-CFIE}
  Let $p$ and $\hat{p}_h$ denote the solutions to \cref{pb:weak-helmholtz-CFIE} and \cref{pb:perturbed-helmholtz-CFIE} for sufficiently small $h$. Then for all $\bs{x}\in\R^3\setminus\overline{\Omega}$
  \begin{align*}
    \vert u(\bs{x}) - u_h(\bs{x})\vert \leq c_{\bs{x}}\left[h^{2m+2}\Vert f\Vert_{m+1} + {\color{black}h^{\ell}\Vert f\Vert_{1} + h^{\ell+1}(\Vert f\Vert_0 + \lVert \nabla F\rVert_{L^{\infty}(\mathcal{N})})}\right],
  \end{align*}
  with $c_{\bs{x}}\to\infty$ as $\bs{x}\to\Gamma$.
\end{theorem}

\begin{proof}
  We write \(u(\bs{x})=v(\bs{x})-i\eta w(\bs{x})\) and \(u_h(\bs{x})=v_h(\bs{x})-i\eta  w_h(\bs{x})\) with
  \begin{align*}
     & v(\bs{x}) = \frac{1}{4\pi}\int_\Gamma(1-ik \vert\bs{x}-\bs{y}\vert)\frac{e^{ik \vert\bs{x}-\bs{y}\vert}}{\vert\bs{x}-\bs{y}\vert^3}(\bs{x}-\bs{y})\cdot\bs{n}(\bs{y})p(\bs{y})d\Gamma(\bs{y}),                                                                                         \\
     & v_h(\bs{x}) = \frac{1}{4\pi}\int_\Gamma(1-ik \vert\bs{x}-\Psi^{-1}_h(\bs{y})\vert)\frac{e^{ik \vert\bs{x}-\Psi^{-1}_h(\bs{y})\vert}}{\vert\bs{x}-\Psi^{-1}_h(\bs{y})\vert^3}\hat{p}_h(\bs{y})(\bs{x}-\Psi^{-1}_h(\bs{y}))\cdot\hat{\bs{n}}_h(\bs{y}) J^{-1}_h(\bs{y}) d\Gamma(\bs{y}), \\
     & w(\bs{x}) =  \frac{1}{4\pi}\int_\Gamma \frac{e^{ik \vert\bs{x}-\bs{y}\vert}}{\vert\bs{x}-\bs{y}\vert}p(\bs{y})d\Gamma(\bs{y}),                                                                                                                                                  \\
     & w_h(\bs{x}) = \frac{1 }{4\pi}\int_\Gamma \frac{e^{ik\vert\bs{x}-\Psi^{-1}_h(\bs{y})\vert}}{\vert\bs{x}-\Psi^{-1}_h(\bs{y})\vert}\hat{p}_h(\bs{y})J^{-1}_h(\bs{y})d\Gamma(\bs{y}).
  \end{align*}
  The difference \(v(\bs{x})-v_h(\bs{x})\) can be bounded with \cref{thm:ptwise-eval-helmholtz-DL}. The difference \(w(\bs{x})-w_h(\bs{x})\) is similar to that in the proof of \cref{thm:ptwise-eval-helmholtz-SL}. More precisely, we split it into
  \begin{align*}
    w(\bs{x}) - w_h(\bs{x}) = & \frac{1}{4\pi}\int_\Gamma\left[\frac{e^{ik\vert\bs{x}-\bs{y}\vert}}{\vert\bs{x}-\bs{y}\vert}p(\bs{y})-\frac{e^{ik\vert\bs{x}-\bs{y}\vert}}{\vert\bs{x}-\bs{y}\vert}\hat{p}_h(\bs{y})\right]d\Gamma(\bs{y})                      \\
                              & + \frac{1}{4\pi}\int_\Gamma\hat{p}_h(\bs{y})\left[\frac{e^{ik\vert\bs{x}-\bs{y}\vert}}{\vert\bs{x}-\bs{y}\vert}-\frac{e^{ik\vert\bs{x} -\Psi^{-1}_h(\bs{y})\vert}}{\vert\bs{x} -\Psi^{-1}_h(\bs{y})\vert}\right]d\Gamma(\bs{y}) \\
                              & + \frac{1}{4\pi}\int_\Gamma\hat{p}_h(\bs{y})\frac{e^{ik\vert\bs{x}-\Psi^{-1}_h(\bs{y})\vert}}{\vert\bs{x} -\Psi^{-1}_h(\bs{y})\vert}(1 - J^{-1}_h(\bs{y}))d\Gamma(\bs{y}).
  \end{align*}
  We bound the first term by $c_{\bs{x}}\Vert p-\hat{p}_h\Vert_{-m-1}$ and the rest by $c_{\bs{x}}h^{\ell+1}\Vert\hat{p}_h\Vert_{0}$, with $c_{\bs{x}}\to\infty$ as $\bs{x}\to\Gamma$.
\end{proof}

\subsection{Wavenumber-explicit estimates for the CFIE}

Recent works have focused on \(k\)-explicit estimates for \cref{pb:helmholtz-CFIE} at high frequency to study the pollution effect, i.e., the fact that keeping \(h k\) constant may not be enough for the error to be bounded independently of \(k\), see~\cite{spence2023}. So far, geometric errors have not been considered in these studies. Here, we use these estimates to derive a \(k\)-explicit bound on the discretization error in the intrinsic norm. Note that in this section, \(c>0\) denotes a generic constant independent of both \(h\) and \(k\).

\begin{lemma}[Consistency]\label{lem:consistency-helmholtz-CFIE-k-explicit}
  There exists \(h_0>0\) such that for all \(h\leq h_0\), the sesquilinear forms defined in \cref{pb:weak-helmholtz-CFIE} and \cref{pb:perturbed-helmholtz-CFIE} satisfy the consistency conditions
  \begin{align*}
    \vert b(\hat{p}_h,\hat{q}_h) - \hat{b}_h(\hat{p}_h,\hat{q}_h)\vert \leq {\color{black}c \eta kh^{\ell+1} \Vert\hat{p}_h\Vert_{0}\Vert\hat{q}_h\Vert_{0} + c_\epsilon kh^{\ell} (1 + k h) \Vert\hat{p}_h\Vert_{\epsilon}\Vert\hat{q}_h\Vert_{0}} \; {\color{black}\forall \epsilon\in(0,1/2),} \, \forall\hat{p}_h,\hat{q}_h\in\hat{V}_h,
  \end{align*}
  with $c_\epsilon\to\infty$ as $\epsilon\to0$.
\end{lemma}
\begin{proof}
We combine \cref{lem:consistency-helmholtz-SL} and \cref{lem:consistency-helmholtz-DL}, making the dependence on the wavenumber $k$ explicit. The single-layer contributions from \cref{lem:consistency-helmholtz-SL} yield error terms of the form
\[
c\, {\color{black}\eta} k h^{\ell+1}\|\hat{p}_h\|_0\|\hat{q}_h\|_0,
\]
whereas the double-layer contributions from \cref{lem:consistency-helmholtz-DL} produce error terms of the form
\[
c_\epsilon\, {\color{black}k h^{\ell} (1 + kh)} \|\hat{p}_h\|_{\epsilon}\|\hat{q}_h\|_0.
\]
The factor {\color{black}$k^2h^{\ell+1}$} arises from the final term in the expansion of $s-s_h$ in the proof of \cref{lem:consistency-helmholtz-DL}:
\[
\frac{ik}{\lvert\Psi_h^{-1}(\bs{x})-\Psi_h^{-1}(\bs{y}) \rvert^2}
\left(
e^{ik\lvert\bs{x}-\bs{y}\rvert}
-
e^{ik\lvert\Psi_h^{-1}(\bs{x})-\Psi_h^{-1}(\bs{y})\rvert}
\right).
\]
{\color{black}The other terms contribute with the factor $kh^{\ell}$, since $\vert s(\bs{x},\bs{y})\vert\leq ck\vert\bs{x}-\bs{y}\vert^{-2}$, which follows from $\vert(1-it)e^{it}-1\vert\leq ct$ for $t\geq0$.}
\end{proof}

\begin{theorem}[Intrinsic norm]\label{thm:intrinsic-norm-helmholtz-CFIE-k-explicit}
  Let \(\Omega\) be a convex domain with strictly positive curvature and analytic boundary \(\Gamma\), and let \(\epsilon\in(0,1/2)\). There exist \(k_0>0\){\color{black}, \(\eta_1>0\)}, \(\eta_0>0\), {\color{black}\(c_0>0\)}, with \(c_0\) depending on \(\epsilon\), such that, if
  \begin{align*}
    {\color{black}k\geq k_0,\quad \eta_1 k\geq \eta\geq k\eta_0, \quad k h \leq c_0, \quad \text{and} \quad kh^{\ell-\epsilon}\leq c_0,}
  \end{align*}
  then \cref{pb:weak-helmholtz-CFIE} and \cref{pb:perturbed-helmholtz-CFIE} have a unique solution and
  \begin{align*}
    \Vert p - \hat{p}_h\Vert_{0} \leq c k^{1/3}\left(h^{\ell+1} ({\color{black}\lVert f \rVert_0 + \lVert \nabla F\rVert_{L^{\infty}(\mathcal{N})}}) + (hk)^{m+1}\lVert p\rVert_{0} + {\color{black}k^2}h^\ell\lVert p\rVert_{0}\right).
  \end{align*}
\end{theorem}
\begin{proof}
We exploit the coercivity of the CFIE {\color{black}in the \(L^2(\Gamma)\)-norm} for sufficiently large $k$ and apply Strang's lemma (\cref{lem:strang-A}) together with $k$-explicit bounds.

We first verify the assumptions of \cref{lem:strang-A}. By \cite[Thm.~2.10]{spence2019}, the continuity condition~\cref{eq:continuity-A} holds with continuity constant $C = O(k^{1/3})$. Moreover, by \cite[Thm.~1.2]{spence2015}, there exist constants $\eta_0>0$ and $k_0>0$ such that, for $k\geq k_0$ and $\eta\geq \eta_0 k$, the coercivity condition~\cref{eq:coercivity} for \cref{pb:weak-helmholtz-CFIE} holds with coercivity constant $\alpha$ independent of $k$.

It remains to establish uniform coercivity for the perturbed Galerkin formulation (\cref{pb:perturbed-helmholtz-CFIE}). {\color{black}By \cref{lem:consistency-helmholtz-CFIE-k-explicit} and an inverse Sobolev inequality, the consistency estimate of \cref{rem:uniform-coercivity} holds in $L^2(\Gamma)\times L^2(\Gamma)$ with $c_h=c_\epsilon kh^{\ell-\epsilon}(1+kh) + c\,\eta kh^{\ell+1}\leq c_\epsilon c_0(1+c_0) + c\,\eta_1c_0^2$, so that \cref{rem:uniform-coercivity} yields \(\beta = \alpha/2\) for $c_0$ sufficiently small (depending on $\epsilon$), independently of $k$ and $h$. Applying Strang's lemma (\cref{lem:strang-A}) with $C\sim k^{1/3}$, and \cref{lem:consistency-helmholtz-CFIE-k-explicit} where the first term is dominated by the second since \(\eta \leq k\eta_1 \) and \(kh\leq c_0\), gives}
\[
\|p-\hat{p}_h\|_{0}
\leq c\, k^{1/3}\Big(
h^{\ell+1}({\color{black}\|f\|_0 + \lVert \nabla F \rVert_{L^{\infty}(\mathcal{N})}})
+ h^{m+1}\|p\|_{m+1}
+ {\color{black}k} h^\ell \|p\|_1
\Big).
\]
Finally, we invoke the regularity estimate $\|p\|_{s}\leq c\,k^{s}\|p\|_0$ from \cite[Thm.~7.2]{spence2025}, with $s=m+1$ for the second term and $s=1$ for the third, to conclude.
\end{proof}

In contrast to~\cite{spence2023}, where geometric discretization errors are not considered, \cref{thm:intrinsic-norm-helmholtz-CFIE-k-explicit} does not allow us to conclude the absence of a pollution effect. Indeed{\color{black}, using \(\lVert \nabla F \rVert_{L^{\infty}(\mathcal{N})}\sim k\) which is typical of plane wave sources,} the $k$-explicit error estimate leads to the more restrictive sufficient conditions
\[
k^{m+4/3} h^{m+1} \le c
\quad \text{and} \quad
{\color{black}k^{7/3}} h^\ell \le c.
\]
These conditions are unlikely to be sharp. The first one is similar to what is obtained when applying C\'ea's lemma (\cref{lem:cea}) as in~\cite{spence2015}, which does not consider geometric discretization error and is known not to be sharp. Indeed, a stronger result is proved in~\cite[Thm.~2.1]{spence2023}. Specifically, there exists a constant $C_{\mathrm{ppw}}$ such that, if $hk \leq C_{\mathrm{ppw}}$ and $k \geq k_0$, then for all $f \in L^2(\Gamma)$ \cref{pb:weak-helmholtz-CFIE} admits a unique Galerkin solution, and the associated quasi-optimality constant is independent of both $k$ and $h$. To the best of our knowledge, no analogous result based on Strang's lemma is currently available.

We therefore refrain from pursuing sharper bounds, such as pointwise error estimates. Nevertheless, the numerical experiments reported in \cref{sec:experiments} suggest that BEMs with curved elements do not exhibit a pollution effect.

\section{Numerical experiments}\label{sec:experiments}

We now present numerical experiments for the Helmholtz equation with Dirichlet boundary conditions. The incident field is the plane wave \( u_{\text{inc}}(\bs{x}) = e^{ik \bs{x} \cdot \bs{d}} \) with \( \bs{d} = (1,0,0) \) and \( k = 2\pi \). The total field \( u = u_{\text{inc}} + u_{\text{scat}} \) is such that \( u_{\text{scat}} = -u_{\text{inc}} \) on \( \Gamma \). We solve for \( u_{\text{scat}} \) using the single-layer and CFIE formulations with \( \eta = k \).

We test polynomial basis functions of degree \( m \in \{0,1,2,3\} \) on curved triangular meshes of
degree \( \ell \in \{1,2,3,4\} \). Accuracy is measured by the pointwise relative error at \( \bs{r} = (1, 2, 3) \),
\[
e_h(\bs{r}) = \frac{|u_h(\bs{r}) - u_{\text{ref}}(\bs{r})|}{|u_{\text{ref}}(\bs{r})|},
\]
where \( u_{\text{ref}} \) is an analytical or reference solution computed on a highly refined mesh.

The singular and nearly-singular integrals in the operator matrices are handled using a
regularization technique based on the density interpolation
method~\cite{faria2021,perezarancibia2019}, adapted to a Galerkin formulation
(similar to \cite{perezarancibia2020}). For computational
efficiency with large systems (up to $10^6$ unknowns), we use classical
$\mathcal{H}$-matrix compression~\cite{hackbusch2015} with a relative tolerance of
$10^{-10}$. The resulting linear systems are solved with GMRES, preconditioned by the
Cholesky factorization of the mass matrix, with a solver tolerance of $10^{-10}$.
The CFIE formulation consistently required a modest, mesh-independent number of iterations ($\approx 30$), while the single-layer formulation needed more iterations ($\approx 300$) on finer meshes{\color{black}, since its condition number grows like $O(h^{-1})$; see~\cite[Sec.~4.5]{sauter2011}}.

We present results for two geometries shown in \cref{fig:geometries}: (i) the unit sphere, enabling comparison with an analytical solution (see, e.g., \cite[eq.~(3.37)]{colton2019}); and (ii) a bean-shaped object, for which we conduct a self-convergence study.

\begin{figure}
  \centering
  \includegraphics[width=0.45\textwidth]{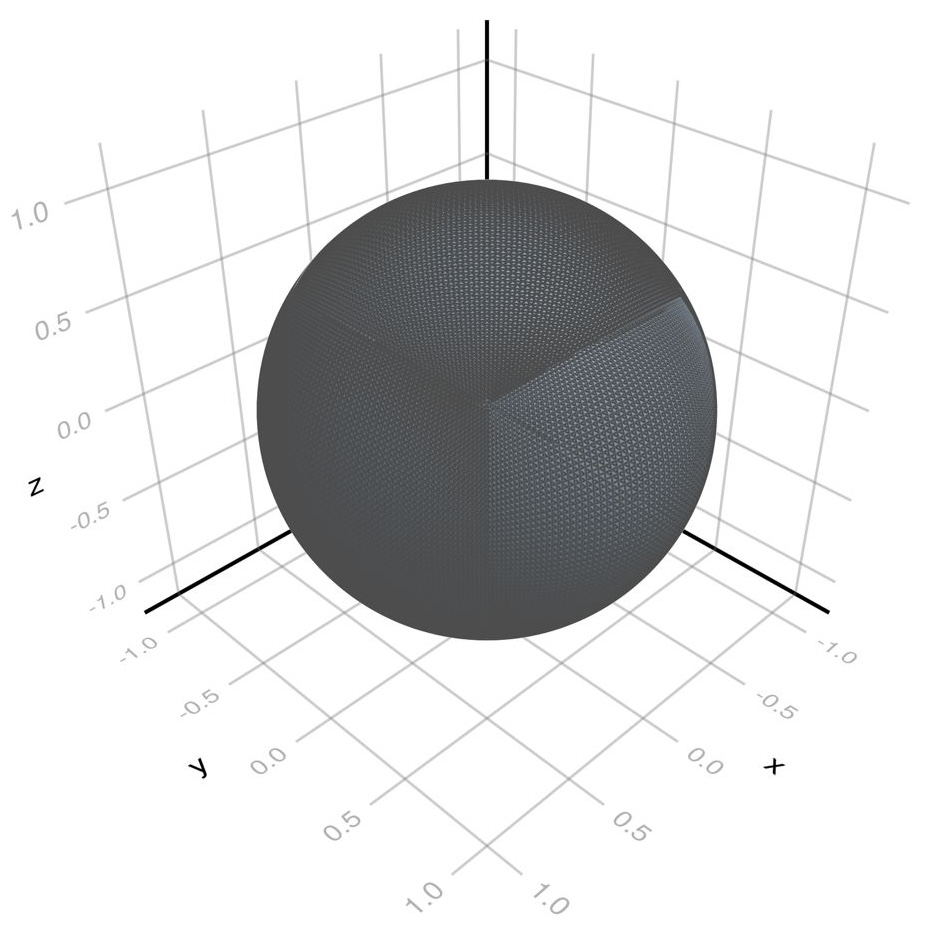}
  \includegraphics[width=0.45\textwidth]{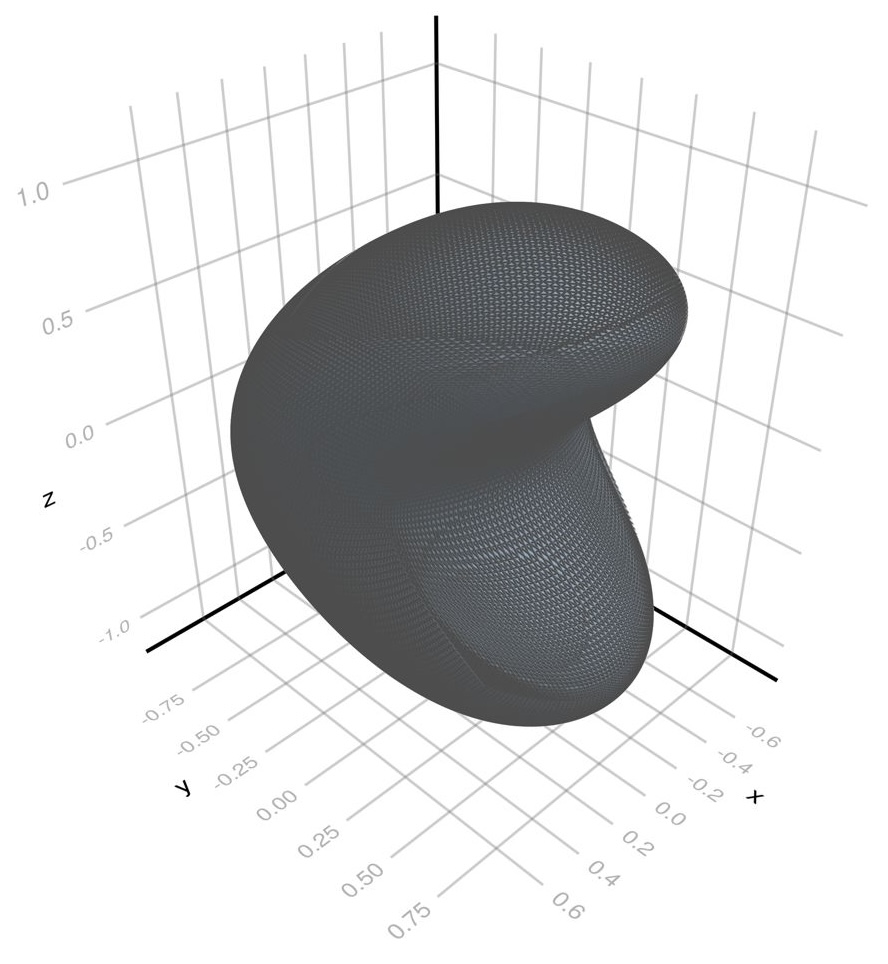}
  \caption{Representative meshes for the sphere and bean-shaped geometries used in the numerical examples. The reference solution for the sphere is obtained via separation of variables, while the reference solution for the bean is computed on a fine mesh.}
  \label{fig:geometries}
\end{figure}

\subsection{Sound-soft sphere}\label{sec:sound-soft-sphere}

Our first test case is the sound-soft sphere, where an exact solution is available via
separation of variables{\color{black}. For the unit sphere centered at the origin, the scattered field reads
\begin{align*}
  u_{\text{scat}}(\bs{x}) = -\sum_{n=0}^{\infty} i^n(2n+1)\,\frac{j_n(k)}{h_n^{(1)}(k)}\,h_n^{(1)}(k\vert\bs{x}\vert)\,P_n\!\left(\frac{\bs{x}\cdot\bs{d}}{\vert\bs{x}\vert}\right),
\end{align*}
where $j_n$ and $h_n^{(1)}$ denote the spherical Bessel and Hankel functions of the first kind, and $P_n$ the Legendre polynomials; see, e.g., \cite[Sec.~3.4]{colton2019}. We truncate the series at $n=N$, which yields the reference value $u_{\text{ref}}(\bs{r})$ to machine precision}. The pointwise errors for a representative set of $\ell$ and $m$
are shown in
\cref{fig:soundsoft-sphere}. In all cases, the observed convergence rates are consistent with the upper bounds
provided by the theory, although we do observe superconvergence of geometrical errors.

For the single-layer formulation, \cref{thm:ptwise-eval-helmholtz-SL} predicts a
convergence rate of $\min(2m+3, \ell+1)$. Such a prediction exactly matches our
numerical results, except for the case of $\ell=2$ and $\ell = 4$, where the geometric
errors appear to converge at the faster rate of $h^{\ell+2}$. The approximation error rate of \(h^{2m+3}\), however, appears to be sharp, as observed for \(m=0\) and \(\ell=2\).

For the CFIE formulation, \cref{thm:ptwise-eval-helmholtz-CFIE} predicts a rate of
$\min(2m+2, \ell)$. We observe, however, geometric errors that converge at the same rate as the
single-layer formulation; that is, at the rate of $h^{\ell+1}$ for odd $\ell$ and
$h^{\ell + 2}$ for even $\ell$. The approximation error rate of \(h^{2m+2}\) appears sharp once again, as confirmed by the case \(m=0\) and \(\ell=2\).

To summarize, a key observation concerns geometric superconvergence for both formulations. 
For the single-layer formulation, the geometric error for even-degree elements converges at the rate \(h^{\ell+2}\), exceeding the predicted \(h^{\ell+1}\). 
For the CFIE formulation, superconvergence occurs at rate \(h^{\ell+1}\) for odd \(\ell\), and at the rate \(h^{\ell+2}\) for even \(\ell\). 
This phenomenon deserves further study.  
These findings are summarized in \cref{tab:experiments}.

\begin{figure}
  \centering
  \includegraphics[width=1\textwidth]{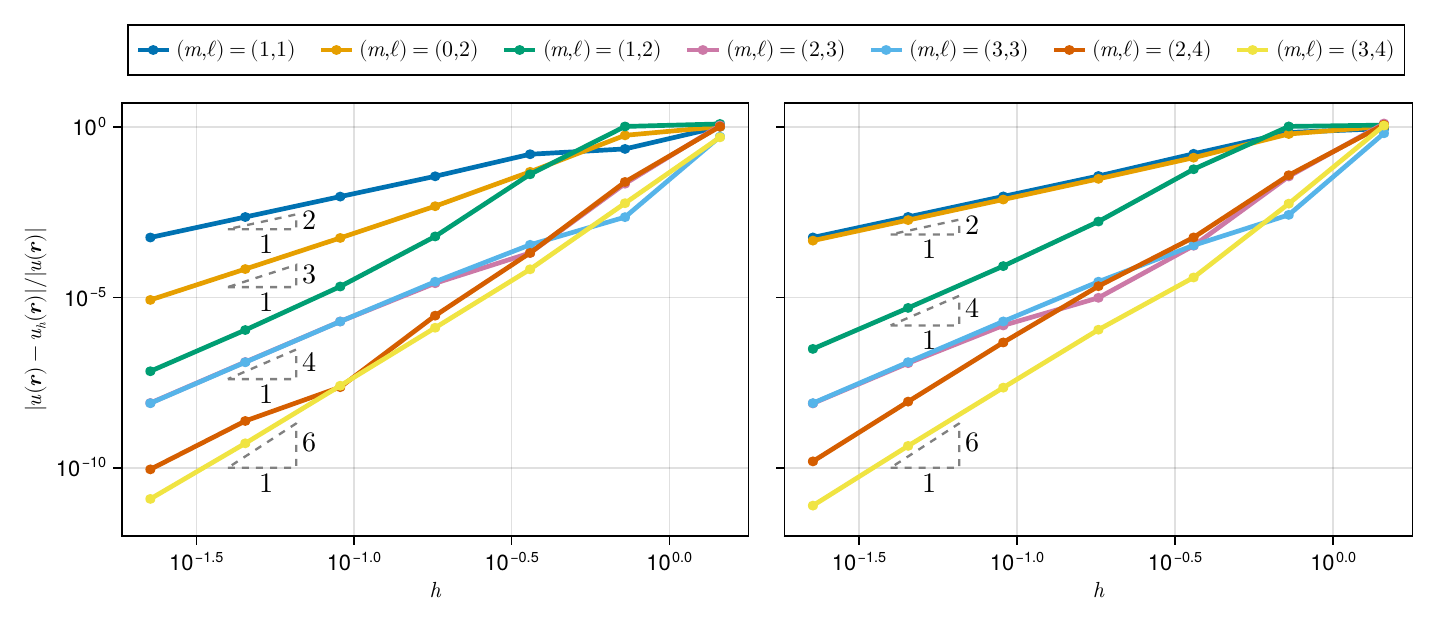}
\caption{Relative error at the observation point $\boldsymbol{r} = (1,2,3)$ for the single-layer (left) and CFIE (right) formulations on the sound-soft sphere. For the single-layer, the observed rate matches the predicted $\min(2m+3, \ell+1)$ when $\ell$ is odd, and improves to $\min(2m+3, \ell+2)$ when $\ell$ is even. For the CFIE, the observed rate exceeds the predicted $\min(2m+2, \ell)$, reaching $\min(2m+2, \ell+1)$ for odd $\ell$ and $\min(2m+2, \ell+2)$ for even $\ell$.}
  \label{fig:soundsoft-sphere}
\end{figure}

\begin{table}
  \caption{\textit{Predicted and observed convergence rates for the single-layer and CFIE formulations for the sound-soft sphere with plane incident wave. We observe super-convergence behavior.}}  
  \centering
  \ra{1.3}
  \begin{tabular}{c|cc}
    \toprule
    & Single-layer                         & CFIE \\
    \midrule
    \multirow{2}{*}{Predicted}
    & \cref{thm:ptwise-eval-helmholtz-SL}  & \cref{thm:ptwise-eval-helmholtz-CFIE} \\
    & $\min(2m+3, \ell+1)$                 & $\min(2m+2, \ell)$                    \\
    \midrule
    \multirow{2}{*}{Observed} 
    & $\min(2m+3, \ell+1)$ for odd $\ell$  & $\min(2m+2, \ell+1)$ for odd $\ell$   \\
    & $\min(2m+3, \ell+2)$ for even $\ell$ & $\min(2m+2, \ell+2)$ for even $\ell$  \\
    \bottomrule
  \end{tabular}
  \label{tab:experiments}
\end{table}

\subsection{Bean-shaped object}

Our second validation employs a bean-shaped object, described in \cite[\S 6.4]{bruno2001}, representing a more complex geometry without an available exact solution. We conduct a self-convergence test, taking as reference solution \(u_{\mathrm{ref}}\) the result computed with the highest-order elements \((m=3, \ell=4)\) on the finest mesh {\color{black} with the CFIE}. The results, displayed in \cref{fig:soundsoft-bean}, exhibit the same convergence behavior previously observed for the sphere. Notably, the geometric superconvergence persists despite the geometry being less symmetric, confirming that this phenomenon is not merely an artifact of the sphere's special symmetries.

\begin{figure}
	\centering
	\includegraphics[width=1\textwidth]{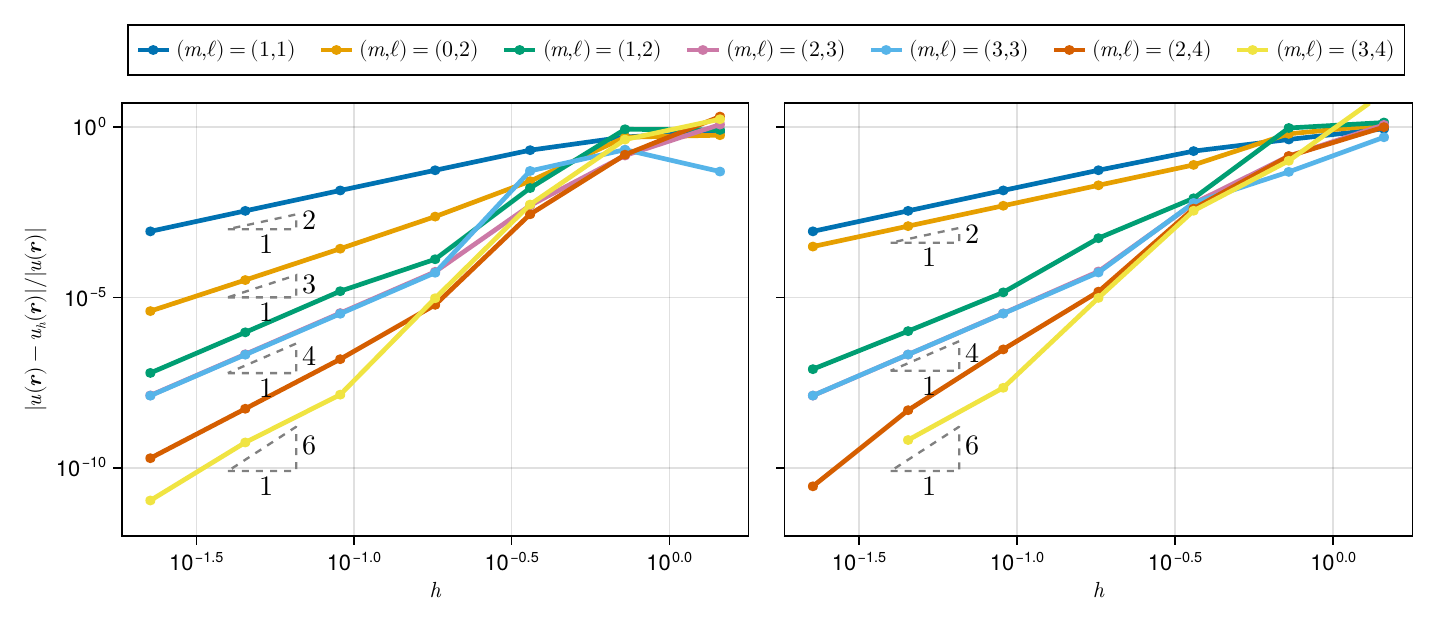}
	\caption{Relative error at an observation point $\boldsymbol{r} = (1,2,3)$ for the single-layer (left) and CFIE formulation (right) for the sound-soft bean problem. The reference solution corresponds to the CFIE formulation with $m=3$ and $\ell=4$ on the finest mesh.}
	\label{fig:soundsoft-bean}
\end{figure}

\subsection{Wavenumber dependency}

Finally, we revisit the sound-soft sphere experiment to
investigate the error dependency on the product $kh$ (where $h$ denotes the mesh size and
$k$ the wavenumber). We use the CFIE formulation (\cref{pb:helmholtz-CFIE}) with
$\eta = k$, fix parameters $m=1$ and $\ell=2$, and vary $k$ from \(\pi\) to
\(16 \pi\). The mesh resolution is examined at three levels, corresponding to \(5\), \(10\),
and \(20\) points per wavelength (ppw).
\textcolor{black}{As in \cref{sec:sound-soft-sphere}, the error is measured against the analytical solution.}

The results in~\cref{tab:error_pivot} confirm the expected absence of pollution
effects. Notably, for a fixed number of points per wavelength, the relative error in the
pointwise evaluation does not grow with~$k$. Instead, it appears to decrease proportionally
with $h$ (as $k$ increases while maintaining fixed $kh$)---a phenomenon that warrants
further investigation. These findings demonstrate that BEMs with curved elements achieve
high accuracy at high frequencies without requiring excessively fine meshes. Furthermore,
they suggest that the theoretical error estimates derived
in~\cref{lem:consistency-helmholtz-CFIE-k-explicit} are overly pessimistic regarding the
explicit wavenumber dependency.

\begin{table}[h]
	\centering
	\caption{Relative error for different wavenumbers $k$ and points per wavelength. We observe numerically that BEMs with curved elements do not suffer from the pollution effect.}
	\label{tab:error_pivot}
	\begin{tabular}{lccccc}
		\toprule
		ppw & $k=\pi$                 & $k=2\pi$                & $k=4\pi$                & $k=8\pi$                & $k=16\pi$               \\
		\midrule
		5   & $4.85 \times 10^{-3}$ & $1.75 \times 10^{-3}$ & $1.12 \times 10^{-3}$ & $6.10 \times 10^{-4}$ & $3.73 \times 10^{-4}$ \\
		10  & $2.37 \times 10^{-4}$ & $1.18 \times 10^{-4}$ & $6.87 \times 10^{-5}$ & $3.47 \times 10^{-5}$ & --                    \\
		20 & $1.78 \times 10^{-5}$ & $8.59 \times 10^{-6}$ & $4.61 \times 10^{-6}$ & --                    & --                    \\
		\bottomrule
	\end{tabular}
\end{table}

\section{Discussion}

In this paper, we have proved sharper convergence rates of boundary element methods for
the 3D Laplace and Helmholtz equations, focusing on smooth geometries and data. We
believe it is important for practitioners to choose $m$ and $\ell$ in a near-optimal
way, especially when repeatedly solving the direct problem to generate synthetic data
for the inverse problem~\cite{montanelli2026, montanelli2023, montanelli2024b}.

Our analysis looked closely at the consistency of the perturbed bilinear and
sesquilinear forms, giving us results that were confirmed by our numerical experiments. Two observations were made.

First, we observed geometric superconvergence of order \(h^{\ell+2}\) for even-degree elements (\(\ell=2,4\)), a phenomenon consistent across both symmetric (sphere) and more complex (bean) geometries. As reported in other contexts~\cite{bonito2018, caubet2024}, this superconvergence appears to be a robust feature rather than an artifact of geometric symmetry. Similar numerical experiments in two dimensions, not shown here, confirmed the same convergence orders as in \cref{fig:soundsoft-sphere,fig:soundsoft-bean}. In particular, this superconvergence was consistently observed and, notably, we managed to break it for \(\ell=2\) on the bean geometry by using a special Lagrange finite element with the edge midpoint interpolation point shifted. This suggests that the finite element analysis presented here is sharp for the single-layer formulation, and that a more detailed analysis---incorporating the specifics of the geometric approximation procedure beyond the polynomial degree \(\ell\)---is necessary.

\textcolor{black}{Second, a geometric superconvergence was also observed for the double-layer and CFIE formulations with the elementwise normal \(\bs{n}_h\), for all values of \(\ell\). Our analysis guarantees a geometric error of order \(h^{\ell}\) for these formulations, whereas we observed \(h^{\ell+1}\), together with superconvergence of order \(h^{\ell+2}\) for even-degree elements. The earlier analyses of N\'{e}d\'{e}lec and Giroire~\cite{giroire1978b, nedelec1977} rely on the interpolated normal \(\bs{\nu}_h\); our experiments suggest that the elementwise normal already suffices. A possible explanation is a partial cancellation between the normal displacement of \(\Gamma_h\) relative to \(\Gamma\) and the error in the elementwise normal, which our consistency estimates do not exploit.}

We plan to investigate these superconvergence phenomena further and extend our analysis to problems such as Maxwell equations and elasticity. For Maxwell equations, pointwise error bounds were obtained in \cite{bendali1984}, but with a loss of order $h^{1/2 + \sigma}$ (for $0 < \sigma \leq 1/2$) due to inverse inequalities. The author suggested that an appropriate Aubin--Nitsche argument could improve these estimates. 

In addition, a central objective of future work is to understand the explicit dependence on the wavenumber $k$; achieving this would likely require a $k$-explicit analogue of Strang's lemma that removes the $k$-dependence of the continuity constant.

\appendix

\section{Theoretical tools}\label{app:theory}

This is largely based on \cite{ern2021b, sauter2011}. For coercive problems, the Lax--Milgram theorem and its discrete counterpart correspond to \cite[Lem.~25.2]{ern2021b} and \cite[Lem.~26.3]{ern2021b}, while C\'{e}a's lemma is given in \cite[Lem.~26.13]{ern2021b}. For inf-sup stable problems, Ne\v{c}as' theorem and its discrete analogue appear in \cite[Thm.~25.9]{ern2021b} and \cite[Thm.~26.6]{ern2021b}, and Babu\v{s}ka's lemma is stated in \cite[Lem.~26.14]{ern2021b}.

\subsection{Coercive case}

Let $V$ be a real or complex Hilbert space, and $a:V\times V\to\R$ (resp.~$\C$) and $F:V\to\R$ (resp.~$\C$) be bilinear (resp.~sesquilinear) and linear (resp.~anti-linear) forms. We assume that $a$ and $F$ are continuous on $V\times V$ and $V$ with constants $C>0$ and $C'>0$, i.e.,
\begin{align}\label{eq:continuity-A}
  \vert a(u,v)\vert \leq C\Vert u\Vert_V \Vert v\Vert_V, \quad \vert F(v)\vert \leq C'\Vert v\Vert_V \quad \forall u,v\in V.
\end{align}
We consider the following abstract weak formulation.

\begin{problem}[Weak formulation]\label{pb:weak-A}
Find $u\in V$ such that
\begin{align*}
  a(u,v) = F(v) \quad \forall v\in V.
\end{align*}
\end{problem}

Assume that $a$ is coercive on $V$, i.e., there exists $\alpha>0$ such that
\begin{align}\label{eq:coercivity}
  \operatorname{Re} a(u,u) \geq \alpha\Vert u\Vert_V^2 \quad \forall u\in V.
\end{align}
Conditions \cref{eq:continuity-A}--\cref{eq:coercivity} imply existence, uniqueness, and stability of the solution via the celebrated 1955 Lax--Milgram theorem \cite{lax1955}.

\begin{theorem}[Lax--Milgram, 1955]\label{thm:lax-milgram}
  Under assumptions \cref{eq:continuity-A}--\cref{eq:coercivity}, there exists a unique solution $u\in V$ to \cref{pb:weak-A}, and
  \begin{align*}
    \Vert u\Vert_V \leq \frac{1}{\alpha}\Vert F\Vert_{V'}.
  \end{align*}
\end{theorem}

We approximate $V$ by a dense sequence $\{V_h\}_{h>0}$ of finite-dimensional subspaces. This gives the following Galerkin approximation problem.

\begin{problem}[Galerkin approximation problem]\label{pb:galerkin-A}
Find $u_h\in V_h$ such that
\begin{align*}
  a(u_h, v_h) = F(v_h) \quad \forall v_h\in V_h.
\end{align*}
\end{problem}

If $a$ is continuous and coercive on $V\times V$, then it is also continuous and coercive on $V_h\times V_h$ with constants $C_h\leq C$ and $\alpha_h\geq\alpha$. Similarly, $F$ is also continuous on $V_h$ with constant $C'_h\leq C'$. Then we have uniqueness to the solution to the Galerkin approximation problem (via \cref{thm:lax-milgram}), and ``quasi-optimality.'' This latter result goes back to Jean C\'{e}a's Ph.D.~thesis in 1964 \cite{cea1964}.

\begin{lemma}[C\'{e}a, 1964] \label{lem:cea}
  Under assumptions \cref{eq:continuity-A}--\cref{eq:coercivity}, there exists a unique solution $u_h\in V_h$ to \cref{pb:galerkin-A}, and
  \begin{align*}
    \Vert u - u_h\Vert_V \leq \frac{C}{\alpha}\inf_{v_h\in V_h}\Vert u - v_h\Vert_V.
  \end{align*}
\end{lemma}

In practice, due to surface approximation with boundary elements, one only has access to some perturbed forms $a_h:V_h\times V_h\to\R$ (resp.~$\C$) and $F_h:V_h\to\R$ (resp.~$\C$). We assume that $a_h$ and $F_h$ are continuous on $V_h\times V_h$ and $V_h$ with constants $D_h>0$ and $D_h'>0$, i.e.,
\begin{align}\label{eq:discrete-continuity}
  \vert a_h(u_h,v_h)\vert \leq D_h\Vert u_h\Vert_V \Vert v_h\Vert_V, \quad \vert F_h(v_h)\vert \leq D_h'\Vert v_h\Vert_V \quad \forall u_h,v_h\in V_h.
\end{align}
This yields the following \textit{perturbed} Galerkin approximation problem.

\begin{problem}[Perturbed Galerkin approximation problem]\label{pb:perturbed-garlerkin-A}
Find $u_h\in V_h$ such that
\begin{align*}
  a_h(u_h, v_h) = F_h(v_h) \quad \forall v_h\in V_h.
\end{align*}
\end{problem}

Assume that $a_h$ is \textit{uniformly} coercive on $V_h$, i.e., there exists $\beta>0$ such that
\begin{align}\label{eq:uniform-coercivity}
  \operatorname{Re} a_h(u_h,u_h) \geq \beta\Vert u_h\Vert_V^2 \quad \forall u_h\in V_h.
\end{align}
Then we have the following result, going back to Strang in 1972 \cite{strang1972}.

\begin{lemma}[Strang, 1972]\label{lem:strang-A}
  Under assumptions \cref{eq:continuity-A}--\cref{eq:uniform-coercivity}, there exists a unique solution $u_h\in V_h$ to \cref{pb:perturbed-garlerkin-A}, and
  \begin{align*}
    \Vert u - u_h\Vert_V \leq c\left\{\sup_{w_h\in V_h}\frac{\vert F(w_h) - F_h(w_h)\vert}{\Vert w_h\Vert_V} + \hspace{-0.05cm} \inf_{v_h\in V_h}\left(\Vert u - v_h\Vert_V + \hspace{-0.05cm} \sup_{w_h\in V_h}\frac{\vert a(v_h,w_h) - a_h(v_h,w_h)\vert}{\Vert w_h\Vert_V}\right)\right\},
  \end{align*}
  with $c=\max\{1 + C/\beta,1/\beta\}$.
\end{lemma}

\begin{remark}[Uniform coercivity]\label{rem:uniform-coercivity}
  If $a$ is coercive and $a_h$ satisfies the consistency estimate
  \begin{align*}
    \vert a(u_h,v_h) - a_h(u_h,v_h)\vert \leq c_h\Vert u_h\Vert_V\Vert v_h\Vert_V \quad \forall u_h,v_h\in V_h,
  \end{align*}
  with $c_h\to0$ as $h\to0$, then $a_h$ is uniformly coercive for sufficiently small $h$ since
  \begin{align*}
    \operatorname{Re} a_h(u_h,u_h) \geq \operatorname{Re} a(u_h,u_h) - \vert a(u_h,u_h) - a_h(u_h,u_h)\vert \geq (\alpha - c_{h_0})\Vert u_h\Vert_V^2 \quad \forall u_h\in V_h, \; \forall h\leq h_0,
  \end{align*}
  for some $h_0>0$.
\end{remark}

\subsection{``Coercive plus compact'' case}

Let $V$ be a complex Hilbert space, and $b:V\times V\to\C$ and $F:V\to\C$ be sesquilinear and anti-linear forms. We assume that $b$ and $F$ are continuous on $V\times V$ and $V$ with constants $C>0$ and $C'>0$, i.e.,
\begin{align}\label{eq:continuity-B}
  \vert b(u,v)\vert \leq C\Vert u\Vert_V \Vert v\Vert_V, \quad \vert F(v)\vert \leq C'\Vert v\Vert_V \quad \forall u,v\in V.
\end{align}
We consider the following abstract weak formulation.

\begin{problem}[Weak formulation]\label{pb:weak-B}
Find $u\in V$ such that
\begin{align*}
  b(u,v) = F(v) \quad \forall v\in V.
\end{align*}
\end{problem}

Assume that there exist a coercive sesquilinear form $a:V\times V\to\C$ and a sesquilinear form $t:V\times V\to\C$, whose associated operator \(T\in L(V,V')\) is compact, such that
\begin{align}\label{eq:garding}
  b(u,v) = a(u,v)+ t(u,v) \quad \forall u,v\in V.
\end{align}
We also assume ``injectivity'' in the first variable, i.e., {\color{black}for every \(u\in V\)}
\begin{align}\label{eq:injectivity}
  {\color{black}(}\forall v \in V,\quad b(u,v)=0 {\color{black})} \implies u=0.
\end{align}
In application of the Fredholm alternative \cite{fredholm1903}, we have the following theorem; see \cite[Thm.~4.2.9]{sauter2011}.

\begin{theorem}[Fredholm, 1903]\label{thm:fredholm}
  Under assumptions \cref{eq:continuity-B}--\cref{eq:injectivity}, there exists a unique solution $u\in V$ to \cref{pb:weak-B}, and
  \begin{align*}
    \Vert u\Vert_V \leq c\Vert F\Vert_{V'}.
  \end{align*}
\end{theorem}

Note that, using the 1962 Ne\v{c}as theorem \cite{necas1962}, the well-posedness obtained in \cref{thm:fredholm} is equivalent to the \textit{continuous} inf-sup conditions
\begin{align*}
   & \inf_{u\in V\setminus\left\{0\right\}}\sup_{v\in V\setminus\left\{0\right\}}\frac{\vert b(u,v)\vert}{\Vert u\Vert_V\Vert v\Vert_V} \geq 1/c > 0, \\
   & \forall v\in V\setminus\left\{0\right\} ,\quad \sup_{u\in V\setminus\left\{0\right\}}\vert b(u,v)\vert>0.
\end{align*}

We approximate $V$ by a dense sequence $\{V_h\}_{h>0}$ of finite-dimensional subspaces. This gives the following Galerkin approximation problem.

\begin{problem}[Galerkin approximation problem]\label{pb:galerkin-B}
Find $u_h\in V_h$ such that
\begin{align*}
  b(u_h, v_h) = F(v_h) \quad \forall v_h\in V_h.
\end{align*}
\end{problem}

We note that $b$ and $F$ are continuous on $V_h\times V_h$ and $V_h$ with constants $C_h\leq C$ and $C'_h\leq C$. However, the \textit{continuous} inf-sup conditions above do not imply the \textit{discrete} inf-sup conditions required for the well-posedness of \cref{pb:galerkin-B}. However, conditions \cref{eq:continuity-B}--\cref{eq:injectivity} actually imply \textit{uniform}, \textit{discrete} inf-sup conditions for sufficiently small $h$~\cite[Thm.~4.2.9]{sauter2011}; that is, there exists $h_0>0$ and $\alpha>0$ such that for all $h \leq h_0$,
\begin{align*}
   & \inf_{u_h \in V_h \setminus \{0\}} \sup_{v_h \in V_h \setminus \{0\}} \frac{|b(u_h, v_h)|}{\|u_h\|_V \|v_h\|_V} \geq \alpha > 0, \\
   & \forall v_h \in V_h \setminus \{0\}, \quad \sup_{u_h \in V_h \setminus \{0\}} |b(u_h, v_h)| > 0.
\end{align*}
Hence, \cref{pb:galerkin-B} is well-posed by the (discrete) Ne\v{c}as theorem. The quasi-optimality result dates back to Babu\v{s}ka's 1971 theorem~\cite{babuska1971}.

\begin{lemma}[Babu\v{s}ka, 1971]\label{lem:babuska}
  Under assumptions \cref{eq:continuity-B}--\cref{eq:injectivity}, there exists a unique solution $u_h\in V_h$ to \cref{pb:galerkin-B}, and
  \begin{align*}
    \Vert u - u_h\Vert_V \leq \left(1+\frac{C}{\alpha}\right)\inf_{v_h\in V_h}\Vert u - v_h\Vert_V.
  \end{align*}
\end{lemma}

Again, in practice, due to surface approximation with boundary elements, one only has access to some perturbed forms $b_h:V_h\times V_h\to\C$ and $F_h:V_h\to\C$. We assume that $b_h$ and $F_h$ are continuous on $V_h\times V_h$ and $V_h$ with constants $D_h>0$ and $D_h'>0$, i.e.,
\begin{align}\label{eq:discrete-continuity-B}
  \vert b_h(u_h,v_h)\vert \leq D_h\Vert u_h\Vert_V \Vert v_h\Vert_V, \quad \vert F_h(v_h)\vert \leq D_h'\Vert v_h\Vert_V \quad \forall u_h,v_h\in V_h.
\end{align}

\begin{problem}[Perturbed Galerkin approximation problem]\label{pb:perturbed-garlerkin-B}
Find $u_h\in V_h$ such that
\begin{align*}
  b_h(u_h, v_h) = F_h(v_h) \quad \forall v_h\in V_h.
\end{align*}
\end{problem}

Assume that the $b_h$ satisfies the discrete inf-sup conditions \textit{uniformly}, i.e., there exists $\beta>0$ such that
\begin{align}
   & \inf_{u_h\in V_h\setminus\left\{0\right\}}\sup_{v_h\in V_h\setminus\left\{0\right\}}\frac{\vert b_h(u_h,v_h)\vert}{\Vert u_h\Vert_V\Vert v_h\Vert_V} \geq \beta > 0, \label{eq:discrete-infsup-unif-a} \\
   & \forall v_h\in V_h\setminus\left\{0\right\} ,\quad \sup_{u_h\in V_h\setminus\left\{0\right\}}\vert b_h(u_h,v_h)\vert>0. \label{eq:discrete-infsup-unif-b}
\end{align}

\begin{lemma}[Strang, 1972]\label{lem:strang-B}
  Under assumptions \cref{eq:continuity-B}--\cref{eq:discrete-infsup-unif-b}, there exists a unique solution $u_h\in V_h$ to \cref{pb:perturbed-garlerkin-B}, and
  \begin{align*}
    \Vert u - u_h\Vert_V \leq c\left\{\sup_{w_h\in V_h}\frac{\vert F(w_h) - F_h(w_h)\vert}{\Vert w_h\Vert_V} + \hspace{-0.05cm} \inf_{v_h\in V_h}\left(\Vert u - v_h\Vert_V + \hspace{-0.05cm} \sup_{w_h\in V_h}\frac{\vert b(v_h,w_h) - b_h(v_h,w_h)\vert}{\Vert w_h\Vert_V}\right)\right\},
  \end{align*}
  with $c=\max\{1 + C/\beta,1/\beta\}$.
\end{lemma}

\begin{remark}[Uniform discrete inf-sup conditions]\label{rem:uniform-infsup}
  If $b$ satisfies the discrete inf-sup conditions uniformly and $b_h$ satisfies the consistency estimate
  \begin{align*}
    \lvert b(u_h,v_h)-b_h(u_h,v_h) \rvert \leq c_h\Vert u_h\Vert_V\Vert v_h\Vert_V \quad \forall u_h,v_h\in V_h,
  \end{align*}
  with $c_h\to0$ as $h\to0$, then the $b_h$ satisfies the discrete inf-sup conditions uniformly for sufficiently small $h$ since
  \begin{align*}
    \vert b_h(u_h,v_h)\vert \geq \vert b(u_h,v_h)\vert - c_{h_0}\Vert u_h\Vert_V\Vert v_h\Vert_V,
  \end{align*}
  implies
  \begin{align*}
     & \inf_{u_h\in V_h\setminus\left\{0\right\}}\sup_{v_h\in V_h\setminus\left\{0\right\}}\frac{\vert b_h(u_h,v_h)\vert}{\Vert u_h\Vert_V\Vert v_h\Vert_V} \geq \alpha - c_{h_0} > 0 \quad \forall h\leq h_0, \\
     & \forall v_h\in V_h\setminus\left\{0\right\} ,\quad \sup_{u_h\in V_h\setminus\left\{0\right\}}\vert b_h(u_h,v_h)\vert>0 \quad \forall h\leq h_0,
  \end{align*}
  for some $h_0>0$.
\end{remark}

\newpage
\section{Geometric estimates and approximation properties}\label{app:geometry}

\subsection{Geometric estimates}

We list useful results, which can be found in \cite[Lems.~2 \& 3]{nedelec1976} and the proof of \cite[Lem.~4.9]{nedelec1977}; see also \cite[Lem.~8.4.11]{sauter2011}, \cite[Lem.~8.4.12]{sauter2011}, and \cite[Lem.~8.4.14]{sauter2011}.

\begin{lemma}\label{lem:geometry}
  Let $k>0$. For all points $\bs{x}$ and $\bs{y}$ on $\Gamma$,
  \begin{align*}
     & \vert 1 - J^{-1}_h(\bs{x})\vert \leq c h^{\ell+1}, \qquad \vert 1 - J^{-1}_h(\bs{y})J^{-1}_h(\bs{x})\vert \leq c h^{\ell+1}, \qquad \vert\bs{x} - \Psi^{-1}_h(\bs{x}) \vert \leq c h^{\ell+1}, \\
     & \vert\bs{n}(\bs{x})-\hat{\bs{n}}_h(\bs{x})\vert \leq c h^{\ell}, \qquad \vert(\bs{x}-\bs{y})-(\Psi^{-1}_h(\bs{x})-\Psi^{-1}_h(\bs{y}))\vert \leq c h^{\ell}\vert\bs{x}-\bs{y}\vert, \\
     & c^{-1}\vert\Psi^{-1}_h(\bs{x})-\Psi^{-1}_h(\bs{y})\vert \leq \vert\bs{x}-\bs{y}\vert \leq c \vert\Psi^{-1}_h(\bs{x})-\Psi^{-1}_h(\bs{y})\vert, \\
     & \left\vert\vert\bs{x}-\bs{y}\vert^{-j} - \vert\Psi^{-1}_h(\bs{x})-\Psi^{-1}_h(\bs{y})\vert^{-j}\right\vert \leq c h^{\ell+1}\vert\bs{x}-\bs{y}\vert^{-j}, \qquad j\in\{1,2,3\}, \\
     & \left\vert \frac{e^{ik\vert\bs{x}-\bs{y}\vert}}{\vert\bs{x}-\bs{y}\vert} - \frac{e^{ik \vert\Psi^{-1}_h(\bs{x})-\Psi^{-1}_h(\bs{y})\vert}}{\vert\Psi^{-1}_h(\bs{x})-\Psi^{-1}_h(\bs{y})\vert}\right\vert \leq c h^{\ell+1}\vert\bs{x} - \bs{y}\vert^{-1}, \\
     & \left\vert e^{ik\vert\bs{x}-\bs{y}\vert}-e^{ik\vert\Psi^{-1}_h(\bs{x})-\Psi^{-1}_h(\bs{y})\vert} \right\vert \leq c h^{\ell+1} \lvert\bs{x}-\bs{y} \rvert.
  \end{align*}
  {\color{black}Moreover, for all $\bs{x}\in\R^3\setminus\Gamma$, all $\bs{y}\in\Gamma$, and sufficiently small $h$,
  \begin{align*}
     & \vert\bs{x}-\Psi^{-1}_h(\bs{y})\vert^{-1} \leq c_{\bs{x}}, \qquad
       \left\vert\vert\bs{x}-\bs{y}\vert^{-j} - \vert\bs{x}-\Psi^{-1}_h(\bs{y})\vert^{-j}\right\vert \leq c_{\bs{x}} h^{\ell+1}, \qquad j\in\{1,3\}, \\
     & \left\vert\frac{e^{ik\vert\bs{x}-\bs{y}\vert}}{\vert\bs{x}-\bs{y}\vert} - \frac{e^{ik\vert\bs{x}-\Psi^{-1}_h(\bs{y})\vert}}{\vert\bs{x}-\Psi^{-1}_h(\bs{y})\vert}\right\vert \leq c_{\bs{x}} h^{\ell+1}, \qquad
       \left\vert e^{ik \lvert \bs{x}-\bs{y}\rvert} - e^{ik \lvert \bs{x}-\Psi^{-1}_h(\bs{y})\rvert}\right\vert \leq c h^{\ell+1}, \\
     & \left\vert \lvert \bs{x} - \bs{y}\rvert e^{ik \lvert \bs{x}-\bs{y}\rvert} - \lvert \bs{x} - \Psi^{-1}_h(\bs{y})\rvert e^{ik \lvert \bs{x}-\Psi^{-1}_h(\bs{y})\rvert}\right\vert\leq c_{\bs{x}}h^{\ell+1},
  \end{align*}
  with $c_{\bs{x}}\to\infty$ as $\bs{x}\to\Gamma$. The constants depend on $k$.}
\end{lemma}

\subsection{Approximation properties}

The results in this section can be found in \cite[Lem.~4]{nedelec1976} and in the proof of \cite[Thm.~4.6]{nedelec1977}.

\begin{lemma}\label{lem:approximation-SL}
  Let $\{V_h\}_{h>0}$ be a dense sequence of finite-dimensional subspaces of $H^{-1/2}(\Gamma_h)$, consisting of continuous Lagrange finite elements of degree \(m\geq1\), or discontinuous Lagrange finite elements of degree \(m=0\), on quasi-uniform meshes \(\Gamma_h\) \cite{ern2021a}. Define
  \begin{align*}
    \hat{V}_h = \{\hat{p}_h=p_h\circ\Psi^{-1}_h,\,p_h\in V_h\} \subset H^{-1/2}(\Gamma).
  \end{align*}
  Let $\hat{s}_h$ be the $L^2(\Gamma)$-orthogonal projector onto $\hat{V}_h$. Then
  \begin{align*}
     & \Vert\hat{p}_h\Vert_{0} \leq c h^{-1/2}\Vert\hat{p}_h\Vert_{-1/2} \quad \forall \hat{p}_h\in\hat{V}_h,  \\
     & \Vert\hat{s}_hp\Vert_{0} \leq \Vert p\Vert_{0} \quad \forall p\in L^2(\Gamma),                          \\
     & \Vert p - \hat{s}_hp\Vert_{0} \leq c h^{m+1}\Vert p\Vert_{m+1} \quad \forall p\in H^{m+1}(\Gamma),      \\
     & \Vert p - \hat{s}_hp\Vert_{-1/2} \leq c h^{m+3/2}\Vert p\Vert_{m+1} \quad \forall p\in H^{m+1}(\Gamma).
  \end{align*}
\end{lemma}

\begin{lemma}\label{lem:approximation-DL}
  Let $\{V_h\}_{h>0}$ be a dense sequence of finite-dimensional subspaces of $L^2(\Gamma_h)$, consisting of continuous Lagrange finite elements of degree \(m\geq1\), or discontinuous Lagrange finite elements of degree \(m=0\), on quasi-uniform meshes \(\Gamma_h\). Define
  \begin{align*}
    \hat{V}_h = \{\hat{p}_h=p_h\circ\Psi^{-1}_h,\,p_h\in V_h\} \subset L^2(\Gamma).
  \end{align*}
  Let $\hat{s}_h$ be the $L^2(\Gamma)$-orthogonal projector onto $\hat{V}_h$. Then {\color{black}for all $\epsilon\in(0,1/2)$}
  \begin{align*}
     & {\color{black}\Vert\hat{p}_h\Vert_{\epsilon} \leq c h^{-\epsilon}\Vert\hat{p}_h\Vert_{0} \quad \forall \hat{p}_h\in\hat{V}_h,} \\
     & \Vert\hat{s}_hp\Vert_{0} \leq \Vert p\Vert_{0} \quad \forall p\in L^2(\Gamma), \\
     & \Vert\hat{s}_hp\Vert_{\epsilon} \leq c\Vert p\Vert_{1} \quad \forall p\in H^1(\Gamma),\\
     & \Vert p - \hat{s}_hp\Vert_{0} \leq c h^{m+1}\Vert p\Vert_{m+1} \quad \forall p\in H^{m+1}(\Gamma), \\
     & {\color{black}\Vert p - \hat{s}_hp\Vert_{\epsilon} \leq c h^{m+1-\epsilon}\Vert p\Vert_{m+1} \quad \forall p\in H^{m+1}(\Gamma).}
  \end{align*}
\end{lemma}

\section*{Acknowledgments}

We are grateful to Marc Lenoir for providing us with his private copy of the technical report \cite{nedelec1977}, and to the staff at \'{E}cole Polytechnique for granting access to the second (and possibly only remaining) physical copy and for digitizing it. We also thank \'{E}ric Luneville for valuable discussions on the convergence of boundary element methods, and Ivan Graham and Houssem Haddar for their encouragement. This paper is dedicated to Jean-Claude N\'{e}d\'{e}lec, in recognition of his inspiring contributions to boundary element methods and numerical analysis.

\bibliographystyle{siam}
\bibliography{_references.bib}

\end{document}